\documentclass[12pt, twoside]{article}
\usepackage{amsmath,amsfonts,amsthm, amssymb}

\usepackage{graphicx}

\usepackage{times}
\usepackage{enumerate}

\usepackage{float}

\restylefloat{figure}

\def\titlerunning#1{\gdef\titrun{#1}}
\makeatletter
\def\author#1{\gdef\autrun{\def\and{\unskip, }#1}\gdef\@author{#1}}
\def\address#1{{\def\and{\\\hspace*{18pt}}\renewcommand{\thefootnote}{}%
\footnote {#1}}%
\markboth{\autrun}{\titrun}}
\makeatother
\def\email#1{e-mail: #1}
\def\subjclass#1{{\renewcommand{\thefootnote}{}%
\footnote{\emph{Mathematics Subject Classification (2010):} #1}}}
\def\keywords#1{\par\medskip
\noindent\textbf{Keywords.} #1}

\newtheorem{lemma}{Lemma}[section]

\newtheorem{proposition}[lemma]{Proposition}
\newtheorem{theorem}[lemma]{Theorem}
\newtheorem{corollary}[lemma]{Corollary}
\theoremstyle{definition}
\newtheorem{definition}[lemma]{Definition}
\newtheorem{example}[lemma]{Example}

\theoremstyle{remark}
\newtheorem{remark}[lemma]{Remark}

\newcommand{\ga}{\gamma}
\newcommand{\la}{\lambda}

\renewcommand{\hat}{\widehat}

\numberwithin{equation}{section} \numberwithin{table}{section}

\newcommand{\diam}{\mathrm{diam}\,}

\newcommand{\vvv}{{|\!|\!|}}

\newcommand{\tr}{\mathrm{tr}\,}

\begin{document}

\baselineskip=17pt

\titlerunning{A Devil's staircase from joint spectral radii}

\title{On a Devil's staircase associated to the joint spectral radii of a family of pairs of matrices}
\author{Ian D. Morris \and Nikita Sidorov}

\date{\today}

\maketitle

\address{Department of Mathematics, University of Surrey, Guildford GU2 7XH, United Kingdom; \email{ian.morris.ergodic@gmail.com}
\and
School of Mathematics, University of Manchester, Oxford Road, Manchester M13 9PL, United Kingdom; \email{sidorov@manchester.ac.uk}}


\begin{abstract}
The joint spectral radius of a finite set of real $d \times d$ matrices is defined to be the maximum possible exponential rate of growth of products of matrices drawn from that set. In previous work with K.~G.~Hare and J.~Theys we showed that for a certain one-parameter family of pairs of matrices, this maximum possible rate of growth is attained along Sturmian sequences with a certain characteristic ratio which depends continuously upon the parameter. In this note we answer some open questions from that paper by showing that the dependence of the ratio function upon the parameter takes the form of a Devil's staircase. We show in particular that this Devil's staircase attains every rational value strictly between $0$ and $1$ on some interval, and attains irrational values only in a set of Hausdorff dimension zero. This result generalises to include certain one-parameter families considered by other authors. We also give explicit formulas for the preimages of both rational and irrational numbers under the ratio function, thereby establishing a large family of pairs of matrices for which the joint spectral radius may be calculated exactly.


\end{abstract}

\section{Introduction}

The spectral radius of a $d \times d$ real matrix $A$, which we denote by $\rho(A)$, is defined to be the maximum of the moduli of the eigenvalues of $A$. If $\|\cdot\|$ is any norm on $\mathbb{R}^d$, then the spectral radius satisfies the well-known identity $\rho(A)=\lim_{n \to \infty}\|A^n\|^{1/n}$. Given a bounded set $\mathsf{A}$ of real $d \times d$ matrices, we by analogy define the \emph{joint spectral radius} of $\mathsf{A}$ to be the quantity
\[\varrho(\mathsf{A}):=\lim_{n\to\infty} \max\left\{\left\|A_{i_n}\cdots A_{i_1}\right\|^{\frac{1}{n}}\colon A_{i_j}\in\mathsf{A}\right\}.\]
It is not difficult to establish that this limit exists (essentially as a consequence of subadditivity) and that its value is independent of the choice of norm $\|\cdot\|$. The joint spectral radius was introduced by G.-C.~Rota and G.~Strang in 1960 (see \cite{RS}, later reprinted in \cite{Rotacoll}) and is the subject of ongoing research interest, which has dealt with its applications, its computation and approximation, and its intrinsic properties as a mathematical function. For a broad range of references on this topic we direct the reader to \cite{Chang,HMST,Jungers,PJB}.

It is not difficult to show that the joint spectral radius admits the alternative formulation
\[
\varrho(\mathsf{A})=\sup_{(A_i)_{i=1}^\infty \in \mathsf{A}^{\mathbb{N}}}\limsup_{n\to\infty} \left\|A_{n}\cdots A_{1}\right\|^{\frac{1}{n}},
\]
and that when $\mathsf{A}$ is compact there exists a sequence $(A_i)$ of elements of $\mathsf{A}$ such that $\|A_n\cdots A_1\|^{1/n} \to \varrho(\mathsf{A})$. (A proof of this statement may be found in \cite{Jungers}.) In this paper we are concerned with the following general question: given a finite set of matrices $\mathsf{A}$ and a sequence $(A_i)$ in $\mathsf{A}$ such that $\|A_n \cdots A_1\|^{1/n} \to \varrho(\mathsf{A})$, what can we say about the structure of the sequence $(A_i)$?

A question of particular interest is that of when there exist \emph{periodic} sequences of matrices which achieve this maximal rate of growth. In \cite{LW}, J.~Lagarias and Y.~Wang asked whether every finite set $\mathsf{A}$ of $d \times d$ real matrices has the property that $\|A_n \cdots A_1\|^{1/n} \to \varrho(\mathsf{A})$ for some periodic sequence of elements of $\mathsf{A}$, or, equivalently, whether every $\mathsf{A}$ has the property that $\varrho(\mathsf{A})=\rho(A_k\cdots A_1)^{1/k}$ for some finite sequence $A_1,\ldots,A_k \in \mathsf{A}$. We shall say that $\mathsf{A}$ has the \emph{finiteness property} if such a periodic sequence exists. The existence of pairs of $2 \times 2$ matrices which do not satisfy the finiteness property was subsequently established by T.~Bousch and J.~Mairesse \cite{BM}, with additional proofs being given later by V.~Blondel, J.~Theys and A.~Vladimirov \cite{BTV} and V.~Kozyakin \cite{Koz3}. The finiteness property continues to be the subject of research investigation: some sufficient conditions for the finiteness property have been given in  \cite{CGSC,Daiquack,DaiKoz,BJ}, and in a recent preprint N. Guglielmi and V. Protasov have given an algorithm for the rigorous verification of the finiteness property for real matrices \cite{GuPr}.

In \cite{HMST}, together with K.~G.~Hare and J.~Theys the present authors investigated the finiteness property for pairs of matrices of the form $\mathsf{A}_\alpha:=\left\{A_0^{(\alpha)}, A_1^{(\alpha)}\right\}$, where
\begin{equation}
\label{hmstmat}
A_0^{(\alpha)}:=\left(\begin{array}{cc}1 & 1 \\ 0 &1 \end{array}\right),\qquad A_1^{(\alpha)}:=\alpha\left(\begin{array}{cc}1 & 0 \\1 & 1 \end{array}\right)
\end{equation}
and $\alpha \in [0,1]$. It was shown in particular that if $(x_i) \in \{0,1\}^{\mathbb{N}}$ is a sequence such that $\|A_{x_n}^{(\alpha)} \cdots A_{x_1}^{(\alpha)}\|^{1/n} \to \varrho(\mathsf{A}_\alpha)$, then the proportion of terms of $(x_i)$ which are equal to $1$ is well-defined and equal to a value $\mathfrak{r}(\alpha) \in [0,1]$ which depends only on $\alpha$. We further showed that $\mathfrak{r}$ is a continuous function of $\alpha$, and gave an explicit expression for a value $\alpha_*$ such that $\mathfrak{r}(\alpha_*) \notin \mathbb{Q}$, providing a completely explicit example of a pair of matrices which does not have the finiteness property (see formula~(\ref{eq:alphastar}) below).

In this paper we undertake a detailed study of the behaviour of the function $\mathfrak{r}$ for $\alpha$ belonging to the larger domain $[0,\infty)$. We extend the results described above in several directions. Firstly we give an explicit formula for $\mathfrak r^{-1}(\gamma)$ when $\gamma \in (0,1) \cap \mathbb{Q}$, and prove that this preimage is always an interval with nonempty interior. This allows us to construct an infinite family of examples of pairs of $2 \times 2$ matrices where the joint spectral radius may be computed exactly. Since the problem of devising algorithms for the computation of the joint spectral radius is ongoing (for some recent contributions see \cite{AAJPR, Chang,GuPr,Kozb2,PJB})  these examples are potentially of value for the testing of new algorithms.

Secondly, we show that the function $\mathfrak{r}$ takes the form of a Devil's staircase, as was conjectured in \cite{HMST,Theys}. The methods which we use to obtain these first two results are significantly more general than those used in \cite{HMST}, and can also be applied to the families of pairs of matrices studied by other authors in \cite{BM,Koz3}. We show in particular that $\mathfrak{r}$ takes rational values only in the complement of a set of Hausdorff dimension zero. This result was previously noted in a special case in \cite{BM}, though no proof was given.

Finally, we give an explicit formula for  $\mathfrak{r}^{-1}(\gamma)$ when $\gamma \in (0,1) \setminus \mathbb{Q}$, and in the special case of the matrices given by \eqref{hmstmat} we provide some inequalities for use in the rigorous computation of its value. We thus show how to  construct an uncountable family of explicit examples for which the finiteness property is not satisfied. As with our explicit description of pairs of matrices which satisfy the finiteness property, we anticipate that these examples may be of value in future in the analysis of algorithms for computing the joint spectral radius.

\section{Notation and statement of results}

Throughout this paper we will consider pairs of real $2 \times 2$ matrices which we denote by $A_0$, $A_1$. To describe the structure of sequences of these matrices we use the space of symbolic sequences $\Sigma_2:=\{0,1\}^{\mathbb{N}}$. We refer to the elements of $\Sigma_2$ as \emph{infinite words}. We equip $\Sigma_2$ with the infinite product topology, with respect to which it is compact and metrisable. On some occasions it will be useful to employ a metric on $\Sigma_2$: to this end, given sequences $(x_i), (y_i) \in \Sigma_2$ we define
\[d[(x_i),(y_i)]:=2^{-\max\{i \colon x_i=y_i \}},\]
where the expression $2^{-\infty}$ is interpreted to mean $0$. This defines an ultrametric on $\Sigma_2$ which generates the infinite product topology. We also define the \emph{shift transformation} $T \colon \Sigma_2 \to \Sigma_2$ by $T[(x_i)]:=(x_{i+1})$, which is a continuous surjection. If a pair of matrices $\mathsf{B}:=\{B_0,B_1\}$ is understood, then following the terminology of \cite{HMST,SAEOJSR} we shall say that a sequence $x=(x_i) \in \Sigma_2$ is \emph{weakly extremal} for $\mathsf{B}$ if $\|B_{x_n}\cdots B_{x_1}\|^{1/n} \to \varrho(\mathsf{B})$ in the limit as $n \to \infty$.

 In addition to considering infinite sequences in $\{0,1\}$ we shall also find it useful to consider finite sequences, which we refer to as \emph{finite words}. If $u = (u_i)_{i=1}^n$ is a finite word we call $n$ the \emph{length} of $u$ and define $|u|:=n$. To simplify certain statements we allow the word of length zero, which we refer to as the \emph{empty word}.

With the pair of matrices $\mathsf{A}=\{A_0,A_1\}$ fixed, we define $A_0^{(\alpha)}:=A_0$ and $A_1^{(\alpha)}:=\alpha A_1$ for all real numbers $\alpha \geq 0$, and let $\mathsf{A}_\alpha:=\{A_0^{(\alpha)},A_1^{(\alpha)}\}$. We shall denote the quantity $\varrho(\mathsf{A}_\alpha)$ simply by $\varrho(\alpha)$. The function $\varrho \colon [0,\infty) \to \mathbb{R}$ is continuous (see for example \cite{HS}). For every $x \in \Sigma_2$, $n \geq 1$ and $\alpha \geq 0$ we define $\mathcal{A}_\alpha(x,n):=A_{x_n}^{(\alpha)} \cdots A_{x_1}^{(\alpha)}$ and $\mathcal{A}(x,n):=\mathcal{A}_1(x,n)=A_{x_n}\cdots A_{x_1}$. If $u$ is a finite word of length $m \geq 1$, we similarly define $\mathcal{A}_\alpha(u)=A^{(\alpha)}_{u_m} \cdots A^{(\alpha)}_{u_1}$ and $\mathcal{A}(u)=\mathcal{A}_1(u)$.

In this paper we are concerned specifically with pairs of matrices such that the maximum growth rate of partial products occurs along \emph{Sturmian} sequences. A large range of definitions of Sturmian sequence exist in the literature, see for example \cite{MH} and the surveys in \cite{PF,Lot}. The definition which we give in this section is not the most straightforward to state, but is the most suited to the proof methods which are used later in this article. In order to state this definition and describe its main consequences, we require some further terminology.

Given a finite word $u$, let $|u|_1$ denote the number of entries of $u$ which are equal to $1$, and if $u$ is not the empty word, define the \emph{slope} of $u$ to be the quantity $\varsigma(u):=|u|_1/|u|$. If $u=(u_i)_{i=1}^n$ and $v=(v_i)_{i=1}^m$ are finite words then we define the \emph{concatenation} of $u$ with $v$, denoted by $uv$, to be the finite word $\omega=(\omega_i)_{i=1}^{n+m}$ such that $\omega_i=u_i$ for $1 \leq i \leq n$ and $\omega_i = v_{i-n}$ for $n < i \leq n+m$. We use the symbols $0$ and $1$ to denote the words of unit length with entries $0$ and $1$ respectively. For positive integers $k$ we use the notation $u^k$ to denote the successive concatenation of $k$ copies of the word $u$, and we define $u^0$ to be the empty word. The word $u^k$ will be referred to as the $k$th power of $u$. Using these notational conventions it is clear that any finite word may be written in the form $1^{a_k}0^{a_{k-1}} \cdots 1^{a_1}$ for some finite collection of non-negative integers $a_i$. Given a finite word $u$ of nonzero length $n$, we use the symbol $u^\infty$ to denote the unique infinite word $x=(x_i)_{i=1}^\infty$ such that $x_{i+kn}=u_i$ for all $k \geq 0$ and $1 \leq i \leq n$.

We say that the finite word $u$ is a \emph{subword} of the finite word $v$ if $v=aub$ for some (possibly empty) finite words $a$ and $b$. If $a$ is empty then we say that $u$ \emph{prefixes} $v$. We shall also say that a finite word $u$ prefixes an infinite word $x \in \Sigma_2$ if $u_i=x_i$ for all $i$ in the range $1 \leq i \leq |u|$. A word $u$ will be called \emph{balanced} if for every pair of subwords $v_1$, $v_2$ of $u$ with $|v_1|=|v_2|$ we have $||v_1|_1-|v_2|_1| \leq 1$. Clearly $u$ is balanced if and only if every subword of $u$ is balanced. We say that $x \in \Sigma_2$ is balanced if every prefix of $x$ is balanced. We say that two finite words $u=(u_i)$, $v=(v_i)$ are \emph{cyclically equivalent} if they are equivalent by some cyclic permutation: that is, they share same length $n$ and there exists an integer $k$ such that $u_i=v_{i+k}$ for $1 \leq i \leq n-k$ and $u_i=v_{i+k-n}$ for $n-k < i \leq n$. It is not difficult to see that $u$ and $v$ are cyclically equivalent if and only if there exist (possibly empty) finite words $a$ and $b$ such that $u=ab$ and $v=ba$. We say that $u$ is \emph{cyclically balanced} if it is balanced and all of its cyclic permutations are also balanced. One may show that a nonempty finite word $u$ is cyclically balanced if and only if $u^\infty$ is balanced (see e.g. \cite[Lemma 4.7]{HMST}).

An infinite word $x \in \Sigma_2$ will be called \emph{Sturmian} if it is balanced and recurrent with respect to $T$. It follows that if $u$ is a finite nonempty word, then $u^\infty$ is Sturmian if and only if $u$ is cyclically balanced. The key properties of Sturmian sequences are outlined by the following theorem, the proof of which may be found in \cite{Lot,MH}.
\begin{theorem}\label{Sturm}
For each $\gamma \in [0,1]$ define a set $X_\gamma \subset \Sigma_2$ as follows: we have $x \in X_\gamma$ if and only if there exists $\delta \in \mathbb{R}$ such that either
\[
x_n \equiv \lfloor \gamma(n+1)+\delta\rfloor - \lfloor \gamma n + \delta\rfloor
\]
or
\[
x_n \equiv \lceil \gamma(n+1)+\delta\rceil - \lceil \gamma n + \delta\rceil,
\]
where $\lfloor x\rfloor = \max\{n\in\mathbb Z : n\le x\}$ and $\lceil x\rceil = \min\{n\in\mathbb Z : n\ge x\}$. Then an infinite word $x \in \Sigma_2$ is Sturmian if and only if $x \in \bigcup_{\gamma \in [0,1]}X_\gamma$. The sets $X_\gamma$ have the following properties:
\begin{enumerate}
\item
Each $X_\gamma$ is compact and satisfies $TX_\gamma=X_\gamma$.
\item
The restriction of $T$ to $X_\gamma$ is uniquely ergodic, i.e., $T$ has a unique invariant measure.
\item
If $x \in X_\gamma$ then $n^{-1}\#\{1 \leq i \leq n \colon x_i=1\} \to \gamma$ as $n \to \infty$.
\item
If $\gamma=p/q$ in least terms then the cardinality of $X_\gamma$ is equal to $q$. If $\gamma$ is irrational then $X_\gamma$ is uncountable.
\end{enumerate}
\end{theorem}
Whilst our primary objective is to continue the study of the pair of matrices defined by \eqref{hmstmat} which were examined in \cite{BTV,HMST,Theys}, the methods which we use are general enough to encompass a larger family. The following definition describes the minimum properties necessary for our arguments to apply:

\begin{definition}\label{leppard}
Let $\mathsf{A}=\{A_0,A_1\}$ be a pair of $2 \times 2$ real matrices. We shall say that $\mathsf{A}$ satisfies the \emph{technical hypotheses} if the following properties hold:
\begin{enumerate}
\item
The matrices $A_0$ and $A_1$ are non-negative, invertible, have positive trace, and do not have a common invariant subspace.
\item
If $u$ is a finite word which is not of the form $1^n$ or $0^n$ then all of the entries of the matrix $\mathcal{A}(u)$ are positive.
\end{enumerate}
We shall further say that $\mathsf{A}$ satisfies the \emph{Sturmian hypothesis} if there exists a function $\mathfrak{r} \colon [0,\infty) \to [0,1]$ such that the following properties hold:
\begin{enumerate}
\setcounter{enumi}{2}
\item
For each $\alpha\geq 0$, every $x \in X_{\mathfrak{r}(\alpha)}$ is weakly extremal for $\mathsf{A}_\alpha$.
\item
For each $\alpha \geq 0$, if $x \in \Sigma_2$ is weakly extremal for $\mathsf{A}_\alpha$ then $n^{-1}\#\{1 \leq i \leq n \colon x_i=1\} \to \mathfrak{r}(\alpha)$.
\item
If $u$ is a finite word which is not cyclically balanced then $\rho(\mathcal{A}_\alpha(u))<\varrho(\alpha)^{|u|}$.
\end{enumerate}
The function $\mathfrak{r}$ will be called the \emph{$1$-ratio function} of the pair $\mathsf{A}$.
\end{definition}

Note that as a consequence of (iv), if $\mathfrak{r}$ exists then it is unique. By \cite[Theorem~2.3]{SAEOJSR} and the minimality of the invariant sets $X_\gamma$, the hypothesis (iii) is in fact equivalent to the hypothesis that $X_{\mathfrak{r}(\alpha)}$ contains at least one extremal infinite word. Some conditions equivalent to (iv) have been used by other authors: a description of these conditions and a proof of their equivalence are given in \cite[\S6]{SAEOJSR}.

A range of examples of pairs $\mathsf{A}$ which satisfy the Sturmian hypothesis are known. In \cite{HMST}, the authors together with K.~G.~Hare and J.~Theys proved that the family of matrices given by \eqref{hmstmat} satisfies parts (iii)-(v) of the Sturmian hypothesis for $\alpha$ restricted to the interval $[0,1]$. If we extend the definition of $\mathfrak{r}$ to the interval $[0,\infty)$ by defining $\mathfrak{r}(\alpha)=1-\mathfrak{r}(1/\alpha)$ for each $\alpha \in (1,\infty)$, then by taking advantage of the relation $A_0=A_1^T$ it is not difficult to show that the Sturmian hypothesis in full for the family $\mathsf{A}_\alpha$. The essential points of this argument are contained in Lemma \ref{switch} below.

In the earlier work \cite{BM}, T.~Bousch and J.~Mairesse also proved that the Sturmian hypothesis holds for the matrices
\begin{equation}\label{eq:bm}
A_0:=\begin{pmatrix} e^{\kappa h_0}+1 & 0 \\ e^\kappa & 1 \end{pmatrix}, \ \
A_1 := \begin{pmatrix} 1 & e^\kappa \\ 0 & e^{\kappa h_1}+1
\end{pmatrix},
\end{equation}
subject to the inequalities $\kappa,h_0, h_1>0$ and $h_0+h_1<2$. Clearly the examples given by \eqref{hmstmat} and \eqref{eq:bm} also satisfy the technical hypotheses.  In a series of papers, V.~S.~Kozyakin has shown that the Sturmian hypothesis holds for pairs of triangular matrices having the form
\[A_0:=\begin{pmatrix} a & b \\ 0 & 1 \end{pmatrix}, \ \
A_1 := \begin{pmatrix} 1 & 0 \\ c & d\end{pmatrix}
\]
where $0<a,d<1 \leq bc$; an overview of this work is given in \cite{Koz3}. Note that the examples considered by Kozyakin satisfy the technical hypotheses in the case $b,c>0$, and are simultaneously similar to a pair of matrices satisfying the technical hypotheses when $b,c<0$.

In order to state the explicit formula for the intervals $\mathfrak{r}^{-1}(p/q)$ which forms part of our first theorem we require one last definition, namely that of a \emph{standard pair}. The set of all standard pairs, which we denote by $\mathcal{P}$, is defined to be the smallest nonempty set of ordered pairs of finite words which satisfies the following two properties: $(0,1) \in \mathcal{P}$, and if $(u,v)\in\mathcal{P}$ then $(uv,v) \in \mathcal{P}$ and $(u,vu)\in\mathcal{P}$. We say that $\omega$ is a \emph{standard word} if it is one half of a standard pair. Every standard word is balanced (see, e.g., \cite[Proposition~2.2.15]{Lot}). If $(u,v)$ is a standard pair then $(u,vu^n)$ and $(uv^n,v)$ are also standard pairs for every $n \geq 0$, and it follows that every power of a standard word is a subword of some standard word. In particular every power of a standard word is balanced, and consequently every standard word is cyclically balanced. A highly detailed analysis of the properties of the set $\mathcal{P}$ may be found in \cite{Lot}.

The main result of the present paper is the following theorem:

\begin{theorem}\label{main}
Let $A_0,A_1$ be a pair of $2 \times 2$ real matrices which satisfies both the technical hypotheses and the Sturmian hypothesis. Then:
\begin{enumerate}
\item
The function $\mathfrak{r}$ is continuous and monotone non-decreasing, and satisfies $\mathfrak{r}(0)=0$ and $\lim_{\alpha \to \infty}\mathfrak{r}(\alpha)=1$.
\item
For every rational number $\gamma \in (0,1)$, the set $\mathfrak{r}^{-1}(\gamma)$ is a closed interval with nonempty interior. The interval $\mathfrak{r}^{-1}(0)\cap (0,\infty)$ is nonempty if and only if $A_0$ is diagonalisable, and similarly $\mathfrak{r}^{-1}(1)$ is nonempty if and only if  $A_1$ is diagonalisable.

The intervals $\mathfrak{r}^{-1}(p/q)$ may be computed exactly by the following procedure. If $A_0$ is diagonalisable, let $P_0:=\lim_{n \to \infty}\rho(A)^{-n}A_0^n$. Then
\[
\mathfrak{r}^{-1}(0)= \left[0,\frac{\rho(A_0)}{\rho(P_0A_1)}\right].
\]
Similarly if $A_1$ is diagonalisable and $P_1:=\lim_{n \to \infty}\rho(A_1)^{-n}A_1^n$, then
\[
\mathfrak{r}^{-1}(1)= \left[\frac{\rho(P_1A_0)}{\rho(A_1)},+\infty\right).
\]
If $p/q \in (0,1)$ in least terms, then there exists a standard pair $(u,v)$ such that $\varsigma(uv)=p/q$. Define $|u|:=q_1$ and $|v|:=q_2$, let $B_1:=\mathcal{A}(u)$, $B_2:=\mathcal{A}(v)$ and $A:=B_1B_2$, and let $P:=\lim_{n \to \infty}\rho(A)^{-n}A^n$ be the Perron projection associated to the positive matrix $A$. Then we have
\[\mathfrak{r}^{-1}(p/q)=\left[\frac{\rho(B_1P)^q} {\rho(A)^{q_1}},\frac{\rho(A)^{q_2}}{\rho(PB_2)^q}\right]\]
and $\varrho(\alpha)=\rho(\mathcal{A}_\alpha(uv))^{1/q}$ for all $\alpha \in \mathfrak{r}^{-1}(p/q)$.
\item
The Hausdorff dimension of the set $\mathfrak{r}^{-1}([0,1)\setminus \mathbb{Q})$ is zero. In particular, $\mathfrak r^{-1}(\gamma)$ is a singleton for any irrational $\gamma\in(0,1)$.
\end{enumerate}
\end{theorem}

Some examples of the explicit formulae generated by Theorem~\ref{main}(ii) are given in Table~\ref{tablemain}. Note that a direct consequence of Theorem~\ref{main}(ii) is that when all of the entries of $A_0$ and $A_1$ are rational, the endpoints of $\mathfrak{r}^{-1}(p/q)$ are algebraic numbers of degree either $1$ or $2$. In the latter case both endpoints belong to the same quadratic field.

As a direct consequence of (iii) we obtain the following result:

\begin{corollary}The function $\mathfrak r$ is not H\"older continuous.
\end{corollary}

To illustrate the behaviour of the function $\mathfrak r$ we reproduce a diagram from \cite{HMST}: see Figure~\ref{fig:frakr(gamma)} below.
\begin{figure}[H]
\[ \includegraphics[width=250pt,height=250pt,angle=270]{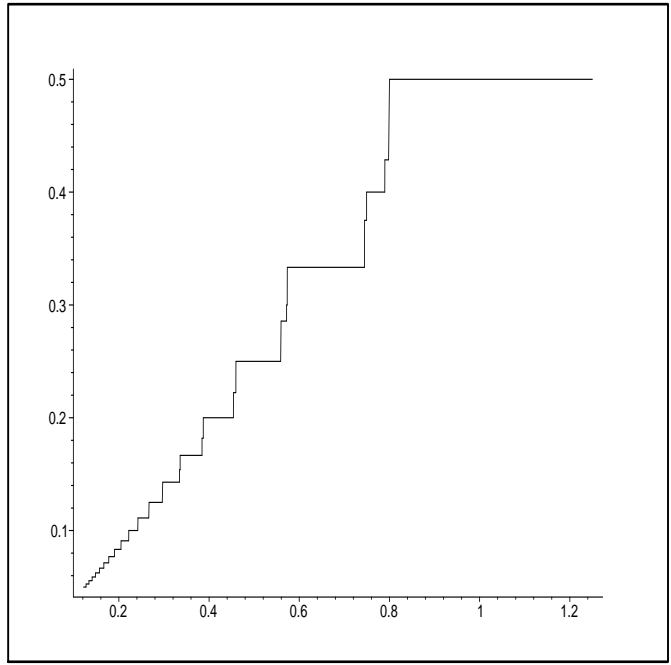} \]
\caption{This figure shows the graph of $\mathfrak{r}$ for $\alpha$ restricted to the interval $[0,5/4]$ for the family of matrices given by \eqref{hmstmat}. Using the explicit formula for $\mathfrak{r}^{-1}(1/n)$ given in Table~\ref{tablemain} one may show that in this case $\alpha^{-1}\mathfrak{r}(\alpha) \to 1/e$ in the limit as $\alpha \to 0$.}
\label{fig:frakr(gamma)}
\end{figure}

\begin{table}\begin{tabular}{|c|c|c|}
\hline
1-ratio $\gamma$ & Standard pair & Interval $\mathfrak{r}^{-1}(\gamma)$\\\hline
$1/2$ & $(0,1)$ & $\left[\frac{4}{5},\frac{5}{4}\right]$\\ \hline
$3/7$ & $(00101,01)$ & $\left[\frac{(5+\frac{127}{168}\sqrt{42})^7} {(13+2\sqrt{42})^5},\frac{(13+2\sqrt{42})^2} {(\frac{3}{2}+\frac{29}{168}\sqrt{42})^7}\right]$ \\ \hline
$2/5$ & $(001,01)$ & $\left[\frac{(2+\frac{17}{24}\sqrt{6})^5} {(5+2\sqrt{6})^3},\frac{(5+2\sqrt{6})^2} {(\frac{3}{2}+\frac{11}{24}\sqrt{6})^5}\right]$\\ \hline
$1/3$ & $(0,01)$ & $\left[\frac{69 - 16\sqrt{3}}{72},\frac{1656-384\sqrt{3}}{1331}\right]$\\ \hline
$2/7$ & $(0001,001)$ & $\left[\frac{(\frac{5}{2}+\frac{41}{40}\sqrt{5})^7} {(9+4\sqrt{5})^4},\frac{(9+4\sqrt{5})^3} {(2+\frac{31}{40}\sqrt{5})^7}\right]$ \\ \hline
$1/4$ & $(0,001)$ & $\left[\frac{496-64\sqrt{21}}{441}, \frac{13671-1764\sqrt{21}}{10000}\right]$\\ \hline
$1/5$ & $(0,0001)$ & $\left[\frac{10612 - 5261\sqrt{2}}{8192},\frac{43466752 - 21549056\sqrt{2}}{28629151}\right]$\\ \hline
$1/6$ & $(0,00001)$ & $\left[\frac{(1+\frac{1}{3\sqrt{5}})^6} {\frac{7}{2}+\frac{3}{2}\sqrt{5}},\frac{(\frac{7}{2}+ \frac{3}{2}\sqrt{5})^{5}}{(3+\frac{19}{15}\sqrt{5})^6}\right]$\\ \hline
$1/7$ & $(0,000001)$ & $\left[\frac{(1+\frac{1}{2\sqrt{15}})^7} {4+\sqrt{15}},\frac{(4+\sqrt{15})^6}{(\frac{7}{2}+ \frac{13}{\sqrt{15}})^7}\right]$\\ \hline
$\frac{1}{n+1}$ & $(0,0^{n-1}1)$ & $ \left[\frac{\left(1+\frac{1}{\sqrt{n^2+4n}}\right)^{n+1}} {1+\frac{n}{2}+\frac{1}{2}\sqrt{n^2+4n}}, \frac{\left(1+\frac{n}{2}+\frac{1}{2}\sqrt{n^2+4n}\right)^n} {\left(\frac{n+1}{2}+ \frac{n^2+3n-2}{2n^2+8n}\sqrt{n^2+4n}\right)^{n+1}}\right]$\\ \hline
\end{tabular}\bigskip
\caption{This table gives some examples of explicit formulae for the intervals $\mathfrak{r}^{-1}(\gamma)$ for the family of pairs of matrices given by \eqref{hmstmat}.  Note that the endpoints of $\mathfrak{r}^{-1}(1/n)$ are asymptotically equal to $e/n + o(1/n)$ in the limit as $n \to \infty$, a feature which may be observed in Figure~\ref{fig:frakr(gamma)}.}
\label{tablemain}
\end{table}

In the cases studied in \cite{BM, HMST, Koz3} the continuity of the function $\mathfrak{r}$ is established by using the particular characteristics of the matrices $A_0,A_1$ in quite a strong fashion. In proving part (i) of Theorem~\ref{main} we observe that the continuity of the function $\mathfrak{r}$ is in fact a corollary of its defining properties. T.~Bousch and J.~Mairesse have asserted in \cite{BM} that part~(iii) of Theorem~\ref{main} holds for the case of triangular matrices of the form~(\ref{eq:bm}), but their proof remains unpublished. In \cite{HMST} we proved for the case of matrices (\ref{hmstmat}) that $\mathfrak r^{-1}(\gamma)$ is a singleton if $\gamma$ is irrational and not Liouville. Theorem~\ref{main}(iii) shows that this remains true for $\gamma$ irrational and Liouville.

In the course of proving part (iii) of Theorem~\ref{main} we are able to establish the following result: if $L\subset (0,1)$ is a compact interval, then there exist constants $K>1$ and $\theta\in (0,1)$ depending on $L$ such that the interval $\mathfrak{r}^{-1}(p/q)$ has diameter less than $K\theta^q$ for every $p/q \in L$, where the fraction $p/q$ is understood to be given in least terms. Heuristically, this result tells us not only that values of $\alpha$ for which $\mathfrak{r}(\alpha)$ is irrational are extremely scarce, but also that values for which $\mathfrak{r}(\alpha)$ is a rational number with large denominator are still relatively scarce, at least when $\alpha$ lies within a given neighbourhood bounded away from zero and infinity. This result is given as Corollary~\ref{rationalbound} below.

The second result of this paper is the following theorem which gives an infinite product formula for $\mathfrak{r}^{-1}(\gamma)$.
\begin{theorem}\label{exact-irrational}
Let $A_0,A_1$ be a pair of matrices which satisfy the technical hypotheses and the Sturmian hypothesis, and let $\gamma \in (0,1) \setminus \mathbb{Q}$. Let $(a_n)_{n=1}^\infty \in \mathbb{N}^{\mathbb{N}}$ be the sequence of continued fraction coefficients of $\gamma$, and for each $n \geq 1$ let $p_n / q_n$ be the corresponding convergent. Define a sequence of finite words $(s_n)$ inductively by setting $s_{-1}:=1$, $s_0:=0$, $s_1:=s_0^{a_1-1}s_{-1}$ and $s_{n+1}:= s_n^{a_{n+1}}s_{n-1}$ for every $n \geq 1$, and for each integer $n \geq -1$ define $\rho_n:=\rho(\mathcal{A}(s_n))$. Then $\mathfrak{r}^{-1}(\gamma)$ is the singleton set whose unique element is given by
\begin{equation}\label{eq:explic}
\alpha_\gamma:=\lim_{n \to \infty}\left(\frac{\rho_n^{q_{n+1}}}{\rho_{n+1}^{q_n}}\right)^{(-1)^n} = \frac{1}{\rho(A_1)}\prod_{n=0}^\infty \left(\frac{\rho_n^{a_{n+1}}\rho_{n-1}}{\rho_{n+1}}\right)^{(-1)^nq_n}.
\end{equation}
\end{theorem}
In general it is not clear whether the infinite product given here will always converge unconditionally, although we are able to prove this in special cases. For the family of matrices defined by \eqref{hmstmat} we are able to give a checkable criterion for a rigorous bound on the error in approximating $\alpha_\gamma$ by partial products of the infinite product given above. The details of these estimates are given in \S8.

\section{Convex analysis and continuity of the 1-ratio}

In this section we give the proof of part~(i) of Theorem~\ref{main}, and introduce a concave function $S \colon [0,1] \to \mathbb{R}$ which characterises the rate of growth of $\mathcal{A}$ along Sturmian trajectories. We begin with the following simple lemma:
\begin{lemma}
\label{switch}
Let $\mathsf{A}=\{A_0,A_1\}$ be a pair of matrices which satisfies the Sturmian hypothesis, and let $\mathfrak{r} \colon [0,\infty) \to [0,1]$ be the corresponding 1-ratio function. Define a new pair of matrices $\hat{\mathsf{A}}:=\{\hat{A}_0,\hat{A}_1\}$ by $\hat{A}_0:=A_1$, $\hat{A}_1:=A_0$. Then $\hat{\mathsf{A}}$ also satisfies the Sturmian hypothesis, and if $\hat{\mathfrak{r}}$ denotes the $1$-ratio function of $\hat{\mathsf{A}}$, then $\mathfrak{r}(\alpha)=1-\hat{\mathfrak{r}}(1/\alpha)$ for all $\alpha \in (0,\infty)$.
\end{lemma}

\begin{proof}
Let us define $\hat{\mathsf{A}}_\alpha:=\{\hat{A}_0,\alpha \hat{A}_1\}$ for each $\alpha \geq 0$ similarly to the definition of $\mathsf{A}_\alpha$, and let $\hat{\varrho}(\alpha)=\varrho(\hat{\mathsf{A}}_\alpha)$ for all $\alpha \geq 0$. Define $\hat{\mathcal{A}}_\alpha(x,n)=\hat{A}^{(\alpha)}_{x_n} \cdots \hat{A}^{(\alpha)}_{x_1}$ for all $x \in \Sigma_2$ and $n \geq 1$, and for each $x =(x_i)\in \Sigma_2$ define a new sequence $\overline{x}$ by $\overline{x_i}:=1-x_i$. We have $x \in X_\gamma$ if and only if $x \in X_{1-\gamma}$. Finally, define $\hat{\mathfrak{r}} \colon [0,\infty) \to [0,1]$ by $\hat{\mathfrak{r}}(\alpha):=1-\mathfrak{r}(1/\alpha)$ for all $\alpha \in (0,\infty)$, and $\hat{\mathfrak{r}}(0):=0$. Note that for all $x \in \Sigma_2$, $n \geq 1$ and $\alpha \in (0,\infty)$ there holds the identity $\hat{\mathcal{A}}_\alpha(x,n) = \alpha^n \mathcal{A}_{1/\alpha}(\overline{x},n)$.
As a direct consequence we have $\hat{\varrho}(\alpha)=\alpha\varrho(1/\alpha)$ for all $\alpha \in (0,\infty)$.

We may now verify directly that $\hat{\mathsf{A}}$ satisfies the Sturmian hypothesis with $\hat{\mathfrak{r}}$ being its $1$-ratio function. The case $\alpha=0$ being trivial, let us fix $\alpha >0$. If $x \in X_{\hat{\mathfrak{r}}(\alpha)}$, then $\overline{x} \in X_{\mathfrak{r}(1/\alpha)}$ and therefore
\[\lim_{n \to \infty} \left\|\hat{\mathcal{A}}_\alpha(x,n)\right\|^{1/n}=\lim_{n \to \infty} \left\|\alpha^n\mathcal{A}_{1/\alpha}(\overline{x},n)\right\|^{1/n} =\alpha \varrho(1/\alpha)=\hat{\varrho}(\alpha)\]
as required. If $\lim_{n \to \infty} \|\hat{\mathcal{A}}_\alpha(x,n)\|^{1/n}= \hat{\varrho}(\alpha)$ then by the same token we have
\[
\lim_{n \to \infty} \|\mathcal{A}_{1/\alpha}(\overline{x},n)\|^{1/n}= \varrho(1/\alpha)
\]
and therefore
\begin{align*}
\lim_{n \to \infty} \frac{1}{n}\#\left\{1 \leq x_j \leq n \colon x_j=1\right\}&=\lim_{n \to \infty} \left(1-\frac{1}{n}\#\left\{1 \leq \overline{x}_j \leq n \colon \overline{x}_j=1\right\}\right)\\
&=1-\mathfrak{r}(1/\alpha)=\hat{\mathfrak{r}}(\alpha) \end{align*}
since Definition~\ref{leppard}(iv) applies to $\mathsf{A}_{1/\alpha}$. Finally, if $u=(u_i)_{i=1}^\ell$ is a finite word which is not cyclically balanced, then the finite word $\overline{u}=(\overline{u}_i)_{i=1}^\ell$ defined by $\overline{u}_i:=1-u_i$ is clearly also not cyclically balanced and hence
\[\rho(\hat{\mathcal{A}}_\alpha(u))= \alpha^{|u|}\rho(\mathcal{A}_{1/\alpha}(\overline{u}))< \alpha^{|u|}\varrho(1/\alpha)^{|u|}= \hat{\varrho}(\alpha)^{|u|}\]
as required. The proof is complete.
\end{proof}

The following general theorem was proved in \cite{SAEOJSR}:

\begin{theorem}\label{gen}
Let $\Delta$ be a metric space and let $A_0, A_1 \colon \Delta \to \mathbf{M}_2(\mathbb{R})$ be continuous functions such that $A_0(\lambda) \neq A_1(\lambda)$ for all $\lambda \in \Delta$. Suppose that there exists a function $\mathfrak{r} \colon \Delta \to [0,1]$ with the following property: for every $x \in \Sigma_2$ such that $ \|A_{x_n}(\lambda) \cdots A_{x_1}(\lambda)\|^{1/n}\to \varrho(\{A_0(\lambda),A_1(\lambda)\})$ in the limit as $n \to \infty$, we have $n^{-1}\{1 \leq i \leq n \colon x_i=1\} \to \mathfrak{r}(\lambda)$. Then the function $\mathfrak{r}$ is continuous.
\end{theorem}
We may now directly deduce several parts of Theorem~\ref{main}(i).
\begin{lemma}\label{lemonsauce}
Let $\mathsf{A}$ be as in Theorem~\ref{main}. Then the 1-ratio function $\mathfrak{r} \colon [0,\infty) \to [0,1]$ is continuous and satisfies $\mathfrak{r}(0)=0$ and $\lim_{\alpha \to \infty} \mathfrak{r}(\alpha)=1$.
\end{lemma}
\begin{proof}
The continuity of $\mathfrak{r}$ follows immediately from Theorem~\ref{gen}. In the case $\alpha=0$ it is obvious that $\|\mathcal{A}_\alpha(x,n)\|^{1/n} \to \rho(A_0)>0$ when $x \in X_0$ and $\|\mathcal{A}_\alpha(x,n)\|^{1/n}\to 0$ for all other $x$, and it follows that $\mathfrak{r}(0)=0$. In particular we have $\lim_{\alpha \to 0}\mathfrak{r}(\alpha)=0$ by continuity. Let $\hat{\mathsf{A}}$ and $\hat{\mathfrak{r}}$ be as in Lemma~\ref{switch}; applying the preceding arguments to $\hat{\mathsf{A}}$ it follows that $\lim_{\alpha \to 0}\hat{\mathfrak{r}}(\alpha)=0$, and therefore
\[\lim_{\alpha \to \infty} \mathfrak{r}(\alpha) = \lim_{\alpha \to \infty} \left(1-\hat{\mathfrak{r}}(1/\alpha)\right)=\lim_{\alpha \to 0} \left(1-\mathfrak{r}(\alpha)\right)=1.\]
\end{proof}

The following proposition, which characterises $\mathfrak{r}$ in terms of a concave function on the unit interval, forms the cornerstone of the proof of Theorem~\ref{main}. In the special case where $A_0$ and $A_1$ are as defined by \eqref{hmstmat}, the results of Proposition~\ref{sfunction} correspond approximately to those of \cite[Proposition~6.1]{HMST}.
\begin{proposition}\label{sfunction}
Let $\mathsf{A}:=\{A_0,A_1\}$ be as in Theorem~\ref{main}. Then there exists a continuous concave function $S \colon [0,1] \to \mathbb{R}$ with the following properties:

\begin{enumerate}
\item
For each $\gamma \in [0,1]$ we have $\frac{1}{n} \log \|\mathcal{A}(x,n)\| \to S(\gamma)$ for every $x \in X_\gamma$.
\item
For each $\alpha \in [0,\infty)$ and $\gamma \in [0,1]$ we have $e^{S(\gamma)}\alpha^\gamma \leq \varrho(\alpha)$
with equality if and only if $\gamma=\mathfrak{r}(\alpha)$. Consequently, for nonzero $\alpha$ we have $\mathfrak{r}(\alpha)=\gamma$ if and only if $-\log\alpha$ is a subgradient of $S$ at $\gamma$.
\item
If $u$ is a word of length $kq$ with $\varsigma(u)=p/q$, then $(kq)^{-1}\log\rho(\mathcal{A}(u))\leq S(p/q)$, with equality if and only if $u$ is cyclically balanced.
\end{enumerate}
\end{proposition}
\begin{proof}

We will show first that there exists a function $S \colon [0,1] \to \mathbb{R}$ such that properties (i)-(iii) hold, and show only at the end of the proof that this function is continuous and concave. We begin by constructing a function $S$ which satisfies (i). Since $X_0$ contains only the single point $0^\infty$ it is clear that (i) holds for $\gamma=0$ with $S(0):=\log\rho(A_0)$, and similarly for $S(1):=\log\rho(A_1)$. Let us therefore consider $\gamma \in (0,1)$ and $x \in X_\gamma$. Using Lemma~\ref{lemonsauce} we may choose $\alpha>0$ such that $\mathfrak{r}(\alpha)=\gamma$. For each $n \geq 1$ we have
\[\log\|\mathcal{A}(x,n)\| = \log\left\|\mathcal{A}_\alpha(x,n)\right\| - \#\left\{1 \leq j \leq n \colon x_j=1\right\}\cdot\log\alpha\]
and it follows by Definition~\ref{leppard}(iii) together with Theorem~\ref{Sturm}(iii) that
\begin{equation}\label{mooo}\lim_{n \to \infty}\frac{1}{n}\log\|\mathcal{A}(x,n)\| = \log\varrho(\alpha)-\gamma\log\alpha.\end{equation}
Since the left hand side of this equation does not depend on the choice of $\alpha \in \mathfrak{r}^{-1}(\gamma)$ and the right hand side does not depend on the choice of $x \in X_\gamma$, we conclude that the identity \eqref{mooo} holds for all such choices. In particular if we define $S(\gamma):=\log\varrho(\alpha)-\gamma\log\alpha$ then (i) is satisfied. Now, if $\alpha \in (0,\infty)$ is given, then for any $\gamma \in [0,1]$ we have for all $x \in X_\gamma$
\begin{align*}
S(\gamma)+\gamma\log\alpha &= \lim_{n \to \infty} \frac{1}{n}\left(\log\left\|\mathcal{A}(x,n)\right\| +\#\left\{1 \leq j \leq n \colon x_j=1\right\}\cdot\log\alpha\right)\\
&=\lim_{n \to \infty}\frac{1}{n}\log\|\mathcal{A}_\alpha(x,n)\| \leq\log \varrho(\alpha).
\end{align*}
It follows from parts~(iii) and (iv) of Definition~\ref{leppard} that the above inequality is an equality if and only if $\gamma=\mathfrak{r}(\alpha)$. This proves (ii) for all cases except those in which $\alpha=0$, which it is trivial to verify directly.

Let us now prove (iii). Let $u$ be as in the statement of the Proposition, and define $x:=u^\infty \in \Sigma_2$.
If $u$ is cyclically balanced then $x$ is Sturmian, and since $\varsigma(u)=p/q$ we necessarily have $x \in X_{p/q}$. Using (i) together with Gelfand's formula we may obtain
\[S(p/q)=\lim_{n \to \infty}\frac{1}{nkq}\log\|\mathcal{A}(x,nkq)\| = \lim_{n \to \infty}\frac{1}{nkq}\log\|\mathcal{A}(u)^n\|= \frac{1}{kq}\log\rho(\mathcal{A}(u)).\]
 Now suppose that $u$ is not cyclically balanced. Using Lemma~\ref{lemonsauce} let us choose $\alpha>0$ such that $\mathfrak{r}(\alpha)=p/q$. Using Definition~\ref{leppard}(v) together with part (ii) above we obtain
\[
\rho(\mathcal{A}(u))\cdot\alpha^{kp} = \rho(\mathcal{A}_\alpha(u))<\varrho(\alpha)^{kq} = e^{kq\cdot S(p/q)}\alpha^{kp}
\]
and therefore $(kq)^{-1}\log\rho(\mathcal{A}(u))<S(p/q)$ as required to prove (iii).

It remains to show that $S$ is continuous and concave. As a consequence of (ii) we have for every $\gamma \in (0,1)$
\[
S(\gamma)=\inf_{\alpha \in (0,\infty)} \log\varrho(\alpha)-\gamma\log\alpha.
\]
The restriction of $S$ to $(0,1)$ is thus an infimum over a set of affine functions of $\gamma$, and hence is concave. It follows from standard results in convex analysis that the restriction of $S$ to $(0,1)$ is also continuous.

Finally let us show that $S$ is continuous on $[0,1]$, and hence is also concave on that interval. We will show that $S$ is continuous at $0$, the case of continuity at $1$ being similar. Since $S$ is concave, the limit of $S$ at $0$ exists, so it suffices to show that there exists a single sequence $(\gamma_n)$ of elements of $(0,1)$ which converges to zero and has the property that $S(\gamma_n)$ converges to $S(0)$. To this end, let us choose a strictly increasing sequence of integers $(n_j)$ such that the sequence $\|A_0^{n_j}\|^{-1}A_0^{n_j}$ converges to some matrix $P$. For each $n \geq 0$ it is not difficult to see that the word $0^n1$ is cyclically balanced, and therefore $S(1/(n+1))=(n+1)^{-1}\log\rho(A_0^nA_1)$ using (iii). We thus have
\begin{align*}
\lim_{j \to \infty}S(1/(n_j+1))-S(0) &= \lim_{j \to \infty}\frac{1}{n_j+1}\log\rho\left(A_0^{n_j}A_1\right) - \log \rho\left(A_0\right)\\
&= \lim_{j \to \infty}\frac{1}{n_j+1}\left(\log\rho\left(A_0^{n_j}A_1\right) - \log \left\|A_0^{n_j}\right\|\right)\\
&= \lim_{j \to \infty}\frac{1}{n_j+1}\log\rho \left({\left\|A_0^{n_j}\right\|}^{-1}A_0^{n_j}A_1\right)=0 \end{align*}
since $\rho(\left\|A_0^{n_j}\right\|^{-1}A_0^{n_j}A_1)$ converges to $\rho(PA_1) $ as $j \to \infty$. The proof is complete.
\end{proof}
The following result together with Lemma~\ref{lemonsauce} completes the proof of Theorem~\ref{main}(i).

\begin{corollary}
The $1$-ratio function $\mathfrak{r}$ is non-decreasing.
\end{corollary}

\begin{proof}
It is an elementary fact in convex analysis that if $\lambda_1$ and $\lambda_2$ are subgradients of a concave function at $\gamma_1$ and $\gamma_2$ respectively, and $\gamma_1<\gamma_2$, then $\lambda_1\ge\lambda_2$. The result now follows by Proposition~\ref{sfunction}.
\end{proof}

\section{Standard words}
In this section we exploit some well-known features of Sturmian words to obtain a pair of propositions dealing with the combinatorial structure of the sets $X_{p/q}$.

Let us define two maps from the set of standard pairs $\mathcal{P}$ to itself by $\Gamma(u,v):=(u,uv)$ and $\Delta(u,v):=(vu,v)$. Throughout the remainder of the paper we use the following notation for continued fractions: if $a_1,\ldots,a_n$ are positive integers, then we use the symbol $[a_1,\ldots,a_n]$ to denote the finite continued fraction
\[
\frac{p_n}{q_n}:=\cfrac{1}{a_1
          + \cfrac{1}{a_2
          + \ldots+\cfrac{1}{a_n}}},
\]
where $p_n$ and $q_n$ are coprime. Given positive integers $a_1,\ldots,a_n$ let us write $p_k/q_k:=[a_1,\ldots,a_k]$ for all $k$ in the range $1 \leq k \leq n$, and define also $p_0:=0$, $q_0:=1$,  $p_{-1}:=1$, $q_{-1}:=0$. Subject to these conventions, for each integer $k$ in the range $1 \leq k \leq n$ the integers $p_k,q_k$ satisfy the recurrence relations $p_k = a_k p_{k-1}+p_{k-2}$ and $q_k = a_k q_{k-1}+q_{k-2}$. If $(a_i)_{i=1}^\infty$ is sequence of positive integers, we use the notation $\gamma = [a_1,a_2,\ldots]$ to mean $\lim_{n \to \infty} [a_1,\ldots,a_n]=\gamma$.

The following proposition will be applied in the proof that $\mathfrak{r}^{-1}([0,1]\setminus \mathbb{Q})$ has zero Hausdorff dimension:

\begin{proposition}
\label{cworp}
Let $\gamma=p_n/q_n = [a_1,\ldots,a_n]$ and $p_{n-1}/q_{n-1}=[a_1,\ldots,a_{n-1}]$ where $n \geq 2$ and $a_n>1$. Then there exist $x \in X_\gamma$ and an integer $k > \frac{1}{3}q_n$ such that $d(x,T^{q_{n-1}}x) \leq 2^{-k}$ and $q_{n-1}^{-1}\log\rho(\mathcal{A}(T^kx,q_{n-1}))= S(p_{n-1}/q_{n-1})$.
\end{proposition}

\begin{proof}
Let $s_{-1}=1$, $s_0=0$ and $s_1:=s_0^{a_1-1}s_{-1}$. Define $s_k$ inductively for $1 < k \leq n$ by $s_k:=s_{k-1}^{a_k}s_{k-2}$. For $k=1$ we have $(s_0,s_1)=\Gamma^{a_1-1}((0,1))$. An easy proof by induction shows that for odd $k>1$,
\[(s_{k-1},s_k)=(\Gamma^{a_k} \circ \Delta^{a_{k-1}} \circ  \cdots \circ \Delta^{a_2}\circ \Gamma^{a_1-1})((0,1))\]
and for even $k$,
\[
(s_k,s_{k-1})=(\Delta^{a_k} \circ \Gamma^{a_{k-1}} \circ  \cdots \circ \Delta^{a_2}\circ \Gamma^{a_1-1})((0,1)),
\]
so in particular each $s_k$ is standard, and hence is cyclically balanced. Define $p_k/q_k:=[a_1,\ldots,a_k]$ for $1 \leq k \leq n$. We have $|s_1|_1=1=p_1$, $|s_1|=a_1=q_1$, and $|s_{k}|_1=a_{k}|s_{k-1}|_1 + |s_{k-2}|_1$ and $|s_{k}|=a_{k}|s_{k-1}|+|s_{k-2}|$ for $1 <k \leq n$. It follows by induction that $|s_k|_1=p_k$ and $|s_k|=q_k$ for $1 \leq k \leq n$. In particular we have $\varsigma(s_n)=p_n/q_n=\gamma$, and since $s_n$ is cyclically balanced it follows that $x:=s_n^\infty \in X_\gamma$. The formula $s_n=s_{n-1}^{a_n}s_{n-2}$ implies that the infinite word $x$ is prefixed by the finite word $s_{n-1}^{a_n}$ and the infinite word $T^{q_{n-1}}x$ is prefixed by the finite word $s_{n-1}^{a_n - 1}$. In particular we have $d(x,T^{q_{n-1}}x) \leq 2^{-k}$ where $k=(a_n-1)q_{n-1}$. We have
\[
q_n=a_nq_{n-1}+q_{n-2} < (a_n+1)q_{n-1}  \leq 3(a_n-1)q_{n-1},
 \]
since $a_n>1$, and thus $k>\frac{1}{3}q_n$ as claimed. Finally we note that $T^kx$ is prefixed by the finite cyclically balanced word $s_{n-1}$ and so by Proposition~\ref{sfunction}(iii) we have
\[
q_{n-1}^{-1}\log \rho(\mathcal{A}(T^kx,q_{n-1}))= |s_{n-1}|^{-1}\log\rho(\mathcal{A}(s_{n-1}))= S(p_{n-1}/q_{n-1}).
\]
The proof is complete.
\end{proof}

The following proposition will be used in the proof that $\mathfrak{r}^{-1}(p/q)$ is an interval:

\begin{proposition}\label{sw2}
Let $p/q \in (0,1)$ with $p$ and $q$ coprime. Then there exists a standard pair $(u,v)$ such that $\varsigma(uv)=p/q$. If $(u,v)$ is such a pair then $|uv|_1=p$, $|uv|=q$, $|u|\cdot|v|_1-|u|_1\cdot|v|=1$, the words $(uv)^nu$ and $(vu)^nv$ are cyclically balanced for all $n \geq 0$, and the word $u^2v^2$ is not cyclically balanced.
\end{proposition}

\begin{proof}
Let us write $p/q=[a_1,\ldots,a_n]$ with $a_n>1$. If $n=1$ then define $(u,v):=\Gamma^{a_n-2}(0,1)$. For odd $n>1$, define
\[(u,v):=(\Gamma^{a_n-1} \circ \Delta^{a_{n-1}} \circ \cdots \circ \Delta^{a_2}\circ \Gamma^{a_1-1})((0,1)),\]
and for even $n$ define
\[(u,v):=(\Delta^{a_n-1} \circ \Gamma^{a_{n-1}} \circ \cdots \circ \Delta^{a_2}\circ \Gamma^{a_1-1})((0,1)).\]
A proof by induction on $n$ similar to that in the previous proposition shows that $|uv|_1=p$ and $|uv|=q$, and hence there exists a standard pair $(u,v)$ such that $\varsigma(uv)=p/q$.

For the rest of the proof we let $(u,v)$ be such a standard pair. Since $(u,v)$ is standard, it follows by definition that the pairs $(\Delta^n\circ\Gamma)(u,v)=((uv)^nu,uv)$ and $(\Gamma^n\circ \Delta)(u,v)=(vu,(vu)^nv)$ are standard pairs for all $n \geq 0$. In particular, the words $(uv)^nu$ and $(vu)^nv$ are standard for all $n \geq 0$, and hence these words are cyclically balanced. It is easy to see that the set of all standard pairs $(a,b)$ such that $|a|\cdot|b|_1-|a|_1\cdot|b|=1$ contains the pair $(0,1)$ and is closed under the action of $\Gamma$ and $\Delta$, and it follows that every standard pair has this property.
In particular we have $|u|\cdot|v|_1-|u|_1\cdot|v|=1$ as claimed. Furthermore, since $(u,uv)$ is a standard pair we have $|u|(|uv|_1)-|u|_1(|uv|)=1$ so that $|uv|$ is coprime to $|uv|_1$, and so any standard pair $(u,v)$ which satisfies $\varsigma(uv)=p/q$ necessarily has $|uv|_1=p$, $|uv|=q$ as claimed.

 Finally, an easy inductive proof starting with the pair $(0,1)$ shows that for every standard pair $(a,b)$ there is a (possibly empty) finite word $p$ such that $ab=p01$ and $ba=p10$, see \cite[p.~57]{Lot}. Since $(u,v)$ is a standard pair we have $uv^2u=p01p10$ for some finite word $p$, and hence in particular $u^2v^2$ is cyclically equivalent to a word of the form $0p01p1$. Since $|1p1|_1 = 2+|0p0|_1$ the word $0p01p1$ is not balanced, and hence $u^2v^2$ is not cyclically balanced.
\end{proof}

\begin{remark}It is worth noting that if we put $p_1=|u|_1, q_1=|u|$ and $p_2=|v|_1, q_2=|v|$, then $p_1/q_1$ and $p_2/q_2$ are the Farey parents of $p/q$: that is, they are the unique fractions such that $p_1/q_1 < p/q< p_2/q_2$, $0<q_1,q_2<q$, $q_1+q_2=q$ and $p_1+p_2=p$. In fact one may show that the pair $(u,v)$ specified by Proposition~\ref{sw2} is unique, but this is not required for our argument.
\end{remark}

\begin{example}Let $p/q=3/7$; then $u=00101, v=01$. In particular we have $u^2v^2=00101001010101$, which contains the subwords $0010100$ and $1010101$ and thus is not balanced (let alone cyclically balanced).
\end{example}

\section{Preimages of rational points}\label{ratchap}

In this section we give the proofs of the various clauses of part~(ii) of Theorem~\ref{main}. Since $S \colon [0,1] \to \mathbb{R}$ is concave it follows from elementary convex analysis that the left derivative $S'_\ell(\gamma)$ and the right derivative $S_r'(\gamma)$ both exist and are finite for every $\gamma \in (0,1)$, the right derivative $S_r'(0)$ at $0$ either exists or equals $+\infty$, and the left derivative $S_\ell'(1)$ at $1$ either exists or equals $-\infty$. Furthermore, the set of all subderivatives of $S$ at $\gamma \in (0,1)$ is precisely $[S'_\ell(\gamma),S_r'(\gamma)]$, the set of subderivatives of $S$ at $0$ is precisely $(-\infty,S'_r(0)]$, and the set of subderivatives of $S$ at $1$ is precisely $[S_\ell'(1),+\infty)$. Note that the latter two intervals are empty if the respective right or left derivative is infinite. For proofs of these statements in a general context we direct the reader to \cite{Rock}.

In Proposition~\ref{sfunction} we showed that $\alpha \in (0,\infty)$ belongs to $\mathfrak{r}^{-1}(\gamma)$ if and only if $-\log\alpha$ is a subgradient of $S$ at $\gamma$. Since $\mathfrak{r}$ is nondecreasing and $\mathfrak{r}(0)=0$, it follows from this result together with the preceding analysis that
\[\mathfrak{r}^{-1}(0) =\left[0,e^{-S_r'(0)}\right],\]
\[\mathfrak{r}^{-1}(1) =\left[e^{-S_\ell'(1)},+\infty\right),\]
and for each $\gamma \in (0,1)$,
\[\mathfrak{r}^{-1}(\gamma) =\left[e^{-S_\ell'(\gamma)},e^{-S_r'(\gamma)}\right],\]
where $e^{-\infty}$ is understood as zero and $e^{+\infty}$ is understood as $+\infty$. Our task in this section, then, is to compute these left and right derivatives explicitly in the case of rational $\gamma$ (showing in the process that they are finite at $0$ and $1$ if and only if the appropriate matrix is diagonalisable) and then show that for $\gamma \in(0,1) \cap\mathbb{Q}$ the left and right derivatives of $S$ at $\gamma$ cannot be equal to one another.

To begin the proof we treat those statements concerned with $\mathfrak{r}^{-1}(0)$ and $\mathfrak{r}^{-1}(1)$. Let us suppose first that $A_0$ is diagonalisable. Since $A_0$ is non-negative it has an eigenvalue equal to its spectral radius (see e.g. \cite[Theorem~8.3.1]{HJ}), and since by Definition~\ref{leppard} its trace is positive, the remaining eigenvalue lies in the interval $(-\rho(A_0),\rho(A_0)]$. It follows easily that the limit $P_0:=\lim_{n \to \infty}\rho(A_0)^{-n}A_0^n$ exists. Using Proposition~\ref{sfunction}(iii) together with the fact that the word $0^n1$ is cyclically balanced, we may now calculate the right derivative of $S$ at $0$ as:
\begin{align*}
S_r'(0)&=\lim_{n \to \infty}\frac{S(1/(n+1))-S(0)}{1/(n+1)}\\
&=\lim_{n \to \infty} (n+1) \left(\frac{1}{n+1}\log \rho(A_0^nA_1)-\log\rho(A_0)\right)\\
&=\lim_{n \to \infty} \log\rho(A_0^nA_1) - (n+1)\log\rho(A_0)\\
&=\lim_{n \to \infty} \log\rho\left(\frac{1}{\rho(A_0)^{n+1}}A_0^nA_1\right)= \log\left(\frac{\rho(P_0A_1)}{\rho(A_0)}\right).
\end{align*}
It follows that $\mathfrak{r}^{-1}(0) = [0,\rho(A_0)/\rho(P_0A_1)]$ as claimed in the statement of Theorem~\ref{main}(ii). Let us now suppose instead that $A_0$ is not diagonalisable. In this case $A_0$ has a repeated eigenvalue equal to its spectral radius and has nontrivial Jordan form. It follows that $\lim_{n \to \infty} \|\rho(A_0)^{-n}A_0^n\| = +\infty$. Let $\delta>0$ be the smallest entry of the matrix $A_0A_1$, which is positive by Definition~\ref{leppard}(ii), and for each $n \geq 2$ let $m_n$ be the largest entry of the non-negative matrix $\rho(A_0)^{-n-1}A_0^{n-1}$. Clearly we have $\lim_{n \to \infty} m_n = +\infty$. Since $A_0^{n-1}$ and $A_0A_1$ are both non-negative matrices it follows easily that
\[2\rho\left(\rho(A_0)^{-n-1}A_0^nA_1\right) \geq \tr \left(\rho(A_0)^{-n-1}A_0^nA_1\right)  \geq \delta m_n\]
for each $n \geq 2$, and hence
\[S_r'(0)=\lim_{n \to \infty}\frac{S(1/(n+1))-S(0)}{1/(n+1)}=\lim_{n \to \infty} \log\rho\left(\frac{1}{\rho(A_0)^{n+1}}A_0^nA_1\right)=+\infty.\]
In this case we therefore have $\mathfrak{r}^{-1}(0)=\{0\}$ as claimed. The proof of the statement concerning $\mathfrak{r}^{-1}(1)$ and the matrix $A_1$ is almost identical, and we omit it for the sake of brevity.

Let us now move on to the case of $\gamma \in (0,1) \cap \mathbb{Q}$. Fix $p/q \in (0,1)$ for the remainder of the proof, and let $(u,v)$ be a standard pair such that $\varsigma(uv)=p/q$, which exists by Proposition~\ref{sw2}. Define $B_1:=\mathcal{A}(u)$, $B_2:=\mathcal{A}(v)$, and $A:=B_1B_2$. By Definition~\ref{leppard}(ii) the matrix $A$ is positive, and hence by the Perron-Frobenius theorem $A$ has two distinct eigenvalues. It follows that the limit
\[
P:= \lim_{n \to \infty} \rho(A)^{-n}A^n
\]
exists, and is the matrix corresponding to the unique projection whose image is the leading eigenspace of $A$ and whose kernel is the non-leading eigenspace of $A$. In particular, $P$ is of rank one.

By Proposition~\ref{sw2} the word $uvuv$ is cyclically balanced and the word $u^2v^2$ is not: since both words have slope $p/q$ and length $2q$, it follows from Proposition~\ref{sfunction}(iii) that
\begin{equation}
\label{xxx}
\rho(B_1B_2)^2= \rho(\mathcal{A}(uvuv))= e^{2qS(p/q)}>\rho(\mathcal{A}(u^2v^2))= \rho(B_1^2B_2^2).
\end{equation}
Using Proposition~\ref{sfunction}(ii) we in particular have $\varrho(\alpha)=e^{S(p/q)}\alpha^{p/q}=\rho(\mathcal{A}_\alpha(uv))^{1/q}$ for every $\alpha \in \mathfrak{r}^{-1}(p/q)$ as claimed in Theorem~\ref{main}(ii). Let $p_1:=|u|_1$, $p_2:=|v|_1$, $q_1:=|u|$, $q_2:=|v|$. Using Proposition~\ref{sfunction}(iii) together with the fact that the words $(uv)^nv$ and $u(uv)^n$ are cyclically balanced, we have for each $n \geq 1$,
\begin{align*}
S\left(\frac{np+p_1}{nq+q_1}\right) &= \frac{1}{nq+q_1}\log\rho(\mathcal{A}(u(uv)^n))= \frac{1}{nq+q_1}\log\rho(B_1A^n),\\
S\left(\frac{np+p_2}{nq+q_2}\right) &= \frac{1}{nq+q_2}\log\rho(\mathcal{A}((uv)^nv))= \frac{1}{nq+q_2}\log\rho(A^nB_2).
\end{align*}
Since $|u|\cdot|v|_1 - |u|_1\cdot|v|=1$ we have $\varsigma(u)<\varsigma(v)$, and therefore $\varsigma((uv)^nu)<\varsigma(uv)<\varsigma((vu)^nv)$ for all $n \geq 0$. We may therefore compute the left derivative of $S$ at $p/q$ as
\begin{align*}
S_\ell'\left(\frac pq\right) &= \lim_{n \to \infty}\frac{S(\frac{p}{q})-S\left(\frac{np+p_1}{nq+q_1}\right)} {\frac{p}{q}-\frac{np+p_1}{nq+q_1}} =\lim_{n \to \infty}\frac{\frac{1}{q}\log\rho(A)-\frac{1}{nq+q_1} \log\rho(B_1A^n)}{\frac{p}{q}-\frac{np+p_1}{nq+q_1}}\\
&= \lim_{n \to \infty}\frac{(nq+q_1)\log \rho(A) - q\log \rho(B_1A^n)}{pq_1-p_1q}\\
&=-\log \left(\frac{\rho(B_1P)^q}{\rho(A)^{q_1}}\right),
\end{align*}
where we have used the identity
\[pq_1-p_1q =(|u|_1+|v|_1)|u| - |u|_1(|u|+|v|)=|u|\cdot|v|_1 - |u|_1\cdot|v|=1.\]
A similar calculation for the right derivative yields
\[
S_r'\left(\frac pq\right) =\log \left(\frac{\rho(PB_2)^q}{\rho(A)^{q_2}}\right).
\]
Combining this with the observations at the start of this section we obtain the explicit formula
\begin{equation}\label{eq:rat-preimage}
\mathfrak r^{-1}\left(\frac pq\right)=\left[e^{-S_\ell'\left(\frac pq\right)},e^{-S_r'\left(\frac pq\right)}\right]=\left[\frac {\rho(B_1P)^q}{\rho(A)^{q_1}}, \frac{\rho(A)^{q_2}}{\rho(PB_2)^q}\right]
\end{equation}
asserted by Theorem~\ref{main}.

\begin{example} Let $A_0, A_1$ be as in (\ref{hmstmat}) and let $p/q=1/2$; here $u=0, v=1$, whence $p_1=0, q_1=1, p_2=1, q_2=1$. Therefore,
\[
B_1=\begin{pmatrix} 1 & 1 \\ 0 & 1\end{pmatrix}, B_2=\begin{pmatrix} 1 & 0 \\ 1 & 1\end{pmatrix},
A=\begin{pmatrix} 2 & 1 \\ 1 & 1\end{pmatrix}, P=\begin{pmatrix} \frac{5+\sqrt5}{10} & \frac{\sqrt5}5 \\ \frac{\sqrt5}5 & \frac{5-\sqrt5}{10}\end{pmatrix},
\]
whence by (\ref{eq:rat-preimage}), $\mathfrak r^{-1}(1/2)=[4/5, 5/4]$. This result was previously derived in \cite{Theys} by a different, geometric method. (See also Table~\ref{tablemain} and Figure~\ref{fig:frakr(gamma)}.)
\end{example}

To finish the proof of Theorem~\ref{main}(ii) we need to show that the interior of $\mathfrak r^{-1}(p/q)$ is nonempty for any $p,q$. Suppose it is empty; then
\begin{equation}\label{contro}
\rho(B_1P)^q\rho(PB_2)^q = \rho(A)^{q_1+q_2}=\rho(A)^q.
\end{equation}
For the remainder of the proof we shall assume that this relation holds, and thereby derive a contradiction.

So, suppose $\rho(B_1P)\rho(PB_2)=\rho(A)$. It follows from the definition of $P$ that $\rho(A)=\rho(PAP)$. Since the spectral radius of a product of matrices is invariant under cyclic permutations of that product, we have $\rho(PB_1)=\rho(B_1P)$ and $\rho(PB_2)=\rho(B_2P)$, and it follows that $\rho(PB_1)\rho(B_2P)=\rho(A)=\rho(PAP)$. Since $A$ is invertible it  has nonzero spectral radius, and therefore $\rho(PB_1)$ and $\rho(B_2P)$ are positive. It follows from the definition of $P$ that $P$ is a non-negative matrix, and hence $PB_1$, $B_2P$ and $PAP$ are all non-negative. In particular, each of these three matrices has an eigenvalue equal to its spectral radius. On the other hand since $P$ has rank one, each of the matrices $PB_1$, $B_2P$ and $PAP$ has determinant zero, and hence has one eigenvalue equal to zero. We conclude that $(\tr PB_1)(\tr B_2P)=\tr PAP$.

Since the two-dimensional matrix $PB_1$ has one positive eigenvalue and one eigenvalue equal to zero, it is diagonalizable. Define $\lambda_1:=\rho(PB_1)$ and $\lambda_2:=\rho(B_2P)$, and choose an invertible matrix $U$ such that
\[
UPB_1U^{-1} = \begin{pmatrix} \lambda_1&0\\0&0\end{pmatrix},\qquad UB_2PU^{-1} = \begin{pmatrix}a&b\\c&d\end{pmatrix}
\]
where $a,b,c$ and $d$ are real numbers. We have $\lambda_2= \tr B_2P = \tr UB_2PU^{-1} = a+d$ and $\lambda_1\lambda_2= (\tr PB_1)(\tr B_2P)=\tr PAP = \tr PB_1B_2P = \tr UPB_1U^{-1}U B_2PU^{-1}= \lambda_1a$. It follows that $d=0$ and $a=\lambda_2$. On the other hand, since $\det P=0$ we have $\det UB_2PU^{-1}=0$ and therefore $ad-bc=0$. We deduce that $bc=0$, and therefore at least one of $b$ and $c$ is zero.

We claim that there exists a nonzero vector $\omega \in \mathbb{R}^2$ which is an eigenvector of both $B_1$ and $B_2$. We consider separately the case $b=0$ and the case $c=0$.

If $b=0$, put $\omega = U^{-1}(0,1)^T$, and note that
\[
PB_1=U^{-1}\begin{pmatrix} \lambda_1 & 0 \\ 0 & 0\end{pmatrix}U,
\]
whence
\[
PB_1\omega = U^{-1}\begin{pmatrix} \lambda_1 & 0 \\ 0 & 0\end{pmatrix}\begin{pmatrix} 0 \\ 1 \end{pmatrix}=\begin{pmatrix} 0 \\ 0 \end{pmatrix},
\]
and
\[
B_2P\omega=U^{-1}\begin{pmatrix} \lambda_2 & 0 \\c & 0\end{pmatrix}\begin{pmatrix} 0 \\ 1 \end{pmatrix}=\begin{pmatrix} 0 \\ 0 \end{pmatrix},
\]
i.e., $PB_1\omega = B_2P\omega=0$.
We claim that $\omega$ is an eigenvector of $B_1$ and of $B_2$. Since $B_2$ is invertible we deduce from the relation $B_2P\omega=0$ that $\omega$ belongs to the kernel of $P$. Since $B_1$ is invertible and $PB_1=0$, we deduce that $B_1\omega$ also
lies in the kernel of $P$, and therefore $\omega$ is an eigenvector of $B_1$. Finally, since by definition the kernel of $P$ is one of the eigenspaces of $A=B_1B_2$, the vector $\omega$ is also an eigenvector of $B_1B_2$, and since $B_1^{-1}\omega$ is proportional to $\omega$ we conclude that $B_2\omega$ is proportional to $\omega$ as required.

The case $c=0$ is similar. Define $\omega := U^{-1}(1,0)^T$ so that $PB_1\omega =\lambda_1\omega$ and $B_2P\omega=\lambda_2\omega$. Since $\omega =\lambda_1^{-1}PB_1\omega$ the vector $\omega$ belongs to the image of $P$, and is therefore fixed by $P$ since $P$ is a projection. It follows from this and the relation $B_2P\omega = \lambda_2\omega$ that
$B_2\omega=\lambda_2\omega$, so that $\omega$ is an eigenvector of $B_2$. On the other hand, the image of $P$ is an eigenspace of $A$ and therefore $\omega$ is an eigenvector of $A=B_1B_2$. Since $B_2\omega=\lambda_2\omega$ we deduce from this that $\omega$ is also an eigenvector of $B_1$ as required. This proves the claim.

We may now derive the desired contradiction. Let $\omega$ be a common eigenvector of the matrices $B_1$ and $B_2$, and let us write $B_i\omega = \xi_i\omega$ for $i=1,2$. We have $B_1^2B_2^2 \omega = \xi_1^2\xi_2^2\omega = (B_1B_2)^2\omega$ so that $\xi_1^2\xi_2^2$ is an eigenvalue of both $B_1^2B_2^2$ and $(B_1B_2)^2$. It follows immediately that $\rho(B_1^2B_2^2)=\max\{|\xi_1^2\xi_2^2|,|(\det B_1B_2)^2\xi_1^{-2}\xi_2^{-2}|\}=\rho(B_1B_2)^2$, contradicting \eqref{xxx}. We conclude that the relation \eqref{contro} cannot hold, and the proof of Theorem~\ref{main}(ii) is complete.

\begin{corollary}\label{strictly-concave}
The function $S:[0,1]\to\mathbb R$ is strictly concave.
\end{corollary}
\begin{proof}We know that $S$ is concave; if it were not strictly concave, there would be an interval $J=(\gamma_1, \gamma_2)$ such that $S|_J$ would be affine, i.e., $S'|_J$ would exist (and be constant). This contradicts the fact that $S'(\gamma)$ does not exist for any rational $\gamma$.
\end{proof}

\begin{remark}
One can show that for the matrices $B_1$ and $B_2$ we have
\[
\rho(B_1B_2)>\rho(B_1)\rho(B_2),
\]
which is essentially equivalent to Corollary~\ref{strictly-concave}. We leave this as an exercise for the interested reader.

\end{remark}

\section{Preimages of irrational points: preliminaries}

In this section we apply a dynamical estimate to prove an inequality dealing with the subgradients of the function $S$. We begin with the following lemma, which allows us to choose a family of norms on $\mathbb{R}^2$ which is adapted to the study of the family $\mathsf{A}_\alpha$. The proof is identical to that of \cite[Lemma~3.3]{HMST}. Note that the proof requires that $A_0$ and $A_1$ do not have a common invariant subspace, as was stipulated in Definition~\ref{leppard}(i).
\begin{lemma}\label{nf}
There exists a family of norms $\{\|\cdot\|_\alpha \colon \alpha \in (0,\infty)\}$ on $\mathbb{R}^2$ with the following properties. For every $v \in \mathbb{R}^2$, $\alpha \in (0,\infty)$ and $i=0,1$ we have $\|A^{(\alpha)}_iv\|_\alpha \leq \varrho(\alpha)\|v\|_\alpha$. If $K\subset (0,\infty)$ is compact, then there is a constant $C>1$ depending on $K$ such that $\|v\| \leq C \|v\|_\alpha \leq C^2\|v\|$ for all $v \in \mathbb{R}^2$ and $\alpha \in K$.
\end{lemma}
For the remainder of this section we fix a family of norms $\|\cdot\|_\alpha$ with the above properties. We also make use of the following elementary result:
\begin{lemma}\label{sv}
Let $B$ be a $2 \times 2$ invertible real matrix, and let $\|\cdot\|$ be the Euclidean norm. Then there exists a rank one matrix $Q$ such that $\|B-Q\| = |\det B| / \|B\|$.
\end{lemma}
\begin{proof}
We apply the existence of a singular value decomposition for $B$. Let us choose unitary matrices $U,V$ and a non-negative diagonal matrix $D$ such that $B=UDV^T$. Since $U$ and $V$ are isometries with respect to the Euclidean norm we have $\|D\|=\|B\|$, and since $|\det U|=|\det V|=1$ we have $|\det D|=|\det B|$. It follows that the nonzero entries of $D$ are $\|B\|$ and $|\det B| / \|B\|$. If $P$ is a matrix which has the entry $\|B\|$ in the same position as for $D$, with all of its other entries being zero, then clearly $P$ has rank one and $\|D-P\|=|\det B|/\|B\|$. Now let $Q:=UPV^T$.
\end{proof}
We now require a dynamical result which describes the dependence of the eigenvectors of certain products $\mathcal{A}_\alpha(u)$ on the structure of the words $u$. The following result is similar in spirit to \cite[Theorem~2.2]{QBWF}, but has the additional property that the modulus of continuity of the vector-valued function depends on $\alpha$ in a controllable manner. The restriction to rational $\gamma$ serves only to simplify the proof: by working instead with two-sided Sturmian sequences indexed over $\mathbb{Z}$, this condition could be removed. The rational case being sufficient for our argument, we ignore the more general statement.
\begin{proposition}\label{invtgraph}
Let $L \subset (0,1)$ be compact, and let $\gamma =\mathfrak{r}(\alpha)\in L \cap \mathbb{Q}$. Then there exist constants $\theta \in (0,1)$ and $K>1$ depending only on $L$ and a function $\mathfrak{v} \colon X_\gamma \to \mathbb{R}^2$ such that the following properties hold. For each $x \in X_\gamma$, $\mathfrak{v}(x)$ is non-negative and satisfies $\|\mathfrak{v}(x)\|_\alpha=1$ and $\mathcal{A}_\alpha(x,n)\mathfrak{v}(x)=\varrho(\alpha)^n\mathfrak{v}(T^nx)$ for all $n \geq 1$. If $x,y \in X_\gamma$ with $d(x,y) \leq 2^{-k}$, then $\|\mathfrak{v}(T^kx)-\mathfrak{v}(T^ky)\|_\alpha \leq K \theta^k$.
\end{proposition}
\begin{proof}
Let us write $\gamma=p/q$ in least terms. Consider any $x \in X_\gamma$: since $0<p/q<1$ the matrix $\mathcal{A}_\alpha(x,q)$ is a mixed product of $A_0$ and $A_1$, and so is positive by Definition~\ref{leppard}(ii). By the Perron-Frobenius theorem it follows that $\mathcal{A}_\alpha(x,q)$ has a unique positive eigenvector with associated eigenvalue equal to $\rho(\mathcal{A}_\alpha(x,q))=\lim_{k \to \infty} \|\mathcal{A}_\alpha(x,kq)\|^{1/k}=\varrho(\alpha)^q$. Let $\mathfrak{v}(x)$ be a positive vector belonging to this eigenspace such that $\|\mathfrak{v}(x)\|_\alpha=1$.

Let us fix $x \in X_\gamma$ and show that $\mathcal{A}_\alpha(x,n)\mathfrak{v}(x)=\varrho(\alpha)^n\mathfrak{v}(T^nx)$ for all $n \geq 1$. If $n \geq q$ then
\begin{align*}
\mathcal{A}_\alpha(x,n)\mathfrak{v}(x) &=\mathcal{A}_\alpha(T^qx,n-q)\mathcal{A}_\alpha(x,q)\mathfrak{v}(x)\\ &=\varrho(\alpha)^q\mathcal{A}_\alpha(T^qx,n-q)\mathfrak{v}(x)\\ &=\varrho(\alpha)^q\mathcal{A}_\alpha(x,n-q)\mathfrak{v}(x), \end{align*}
and by iterating this identity we may reduce to the case where $1 \leq n <q$. In this case we have
\begin{align*}\varrho(\alpha)^q = \|\mathcal{A}_\alpha(x,q)\mathfrak{v}(x)\|_\alpha &= \|\mathcal{A}_\alpha(T^nx,q-n)\mathcal{A}_\alpha(x,n)\mathfrak{v}(x)\|_\alpha \\
&\leq \|\mathcal{A}_\alpha(T^nx,q-n)\|_\alpha \|\mathcal{A}_\alpha(x,n)\mathfrak{v}(x)\|_\alpha\\
&\leq\varrho(\alpha)^{q-n}\|\mathcal{A}_\alpha(x,n)\mathfrak{v}(x)\|_\alpha \leq \varrho(\alpha)^q\end{align*}
and it follows that $\|\mathcal{A}_\alpha(x,n)\mathfrak{v}(x)\|_\alpha=\varrho(\alpha)^n$. Moreover we have
\begin{align*}
\mathcal{A}_\alpha(T^nx,q)\mathcal{A}_\alpha(x,n)\mathfrak{v}(x)
&=
\mathcal{A}_\alpha(T^qx,n)\mathcal{A}_\alpha(x,q)\mathfrak{v}(x)\\
&=\varrho(\alpha)^q\mathcal{A}_\alpha(T^qx,n)\mathfrak{v}(x)=
\varrho(\alpha)^q\mathcal{A}_\alpha(x,n)\mathfrak{v}(x).\end{align*}
Thus $\mathcal{A}_\alpha(x,n)\mathfrak{v}(x)$ is a positive eigenvector of $\mathcal{A}_\alpha(T^nx,q)$ with corresponding eigenvalue $\varrho(\alpha)^q$ and with norm equal to $\varrho(\alpha)^n$. Using the uniqueness of the leading eigenspace in the Perron-Frobenius theorem we conclude that $\varrho(\alpha)^{-n}\mathcal{A}_\alpha(x,n)\mathfrak{v}(x)=\mathfrak{v}(T^nx)$ as claimed.

Since $\mathfrak{r}$ is continuous and monotone with $\mathfrak{r}(0)=0$ and $\lim_{\alpha \to \infty}\mathfrak{r}(\alpha)=1$, the set $\mathfrak{r}^{-1}(L)$ is a closed subset of $(0,\infty)$ which is bounded away from $0$ and $\infty$, hence compact. For every $\alpha \in \mathfrak{r}^{-1}(L)$ we have $\max\{\rho(A_0),\rho(\alpha A_1)\}<\varrho(\alpha)$ by Definition~\ref{leppard}(v), since $\mathfrak{r}(\alpha) \notin \{0,1\}$. Let us define
\[\theta:=\sup_{\alpha \in \mathfrak{r}^{-1}( L)} \max\left\{\frac{|\det A_0|}{\varrho(\alpha)^2},\frac{|\det (\alpha A_1)|}{\varrho(\alpha)^2}\right\} \leq \sup_{\alpha\in\mathfrak{r}^{-1}(L)} \frac{\max\{\rho(A_0)^2,\rho(\alpha A_1)^2\}}{\varrho(\alpha)^2}<1,\]
and let $x,y \in X_\gamma$ with $d(x,y)\leq 2^{-k}$. We will show that $\|\mathfrak{v}(T^kx)-\mathfrak{v}(T^ky)\|_\alpha \leq 6C^6\theta^k$, where $C>1$ is the constant provided by Lemma~\ref{nf} with respect to the compact set $\mathfrak{r}^{-1}(L)$. If $3C^6\theta^k \geq 1$ then clearly $\| \mathfrak{v}(x)-\mathfrak{v}(y)\|_\alpha \leq 2 \leq 6C^6\theta^k$, so we shall assume for the remainder of the proof that $3C^6\theta^k<1$. Since
\[\frac{|\det (\varrho(\alpha)^{-k}\mathcal{A}_\alpha(x,k))|} {\|\varrho(\alpha)^{-k}\mathcal{A}_\alpha(x,k)\|} \leq \frac{C^2 \theta^k}{\| \varrho(\alpha)^{-k}\mathcal{A}_\alpha(x,k)\|_\alpha} =C^2\theta^k,\]
it follows from Lemma~\ref{sv} that there exists a rank one matrix $Q \in \mathbf{M}_2(\mathbb{R})$ such that $\|\varrho(\alpha)^{-k}\mathcal{A}_\alpha(x,k)-Q\| \leq C^2\theta^k$. Since $d(x,y) \leq 2^{-k}$ we have $\mathcal{A}_\alpha(x,k)=\mathcal{A}_\alpha(y,k)$ and therefore also $\|\varrho(\alpha)^{-k}\mathcal{A}_\alpha( y,k)-Q\| \leq C^2\theta^k$. Clearly,
\[\|\varrho(\alpha)^{-k}\mathcal{A}_\alpha(x,k)\mathfrak{v}(x) - Q\mathfrak{v}(x)\|_\alpha \leq C^4\theta^k<\frac{1}{3C^2}<\frac{1}{3},\]
and therefore in particular $|1-\|Q\mathfrak{v}(x)\|_\alpha| <1/3$ so that $2/3 < \|Q\mathfrak{v}(x)\|_\alpha< 4/3$. By identical reasoning we have $2/3<\| Q\mathfrak{v}(y)\|_\alpha<4/3$. Now, since the image of $Q$ is one-dimensional, there exists $\lambda \in \mathbb{R}$ such that $Q\mathfrak{v}(x)=\lambda Q\mathfrak{v}(y)$. We have $|\lambda|=\|Q\mathfrak{v}(x)\|_\alpha / \|Q\mathfrak{v}(y)\|_\alpha < 2$ and therefore
\begin{align*}\|\mathfrak{v}(T^kx)-\lambda \mathfrak{v}(T^ky)\|_\alpha &=\|\varrho(\alpha)^{-k}\mathcal{A}_\alpha(x,k)\mathfrak{v}(x) - \lambda\varrho(\alpha)^{-k}\mathcal{A}_\alpha(y,k)\mathfrak{v}(y)\|_\alpha\\
&\leq \|\varrho(\alpha)^{-k}\mathcal{A}_\alpha(x,k)\mathfrak{v}(x) - Q\mathfrak{v}(x)\|_\alpha\\
 & + \|\lambda Q\mathfrak{v}(y) - \lambda\varrho(\alpha)^{-k}\mathcal{A}_\alpha(y,k)\mathfrak{v}(y)\|_\alpha\\
&\leq (1+|\lambda|)C^4\theta^k <3C^4\theta^k.\end{align*}
We now claim that the real number $\lambda$ is positive. Suppose that it is negative: since $\mathfrak{v}(T^kx)$ and $\mathfrak{v}(T^ky)$ are non-negative, we must have $\|\mathfrak{v}(T^kx)-\lambda \mathfrak{v}(T^ky)\| \geq \|\mathfrak{v}(T^kx)\|$ and therefore
\begin{align*}\|\mathfrak{v}(T^kx) - \lambda \mathfrak{v}(T^kx)\|_\alpha &\geq C^{-1}\|\mathfrak{v}(T^kx)-\lambda \mathfrak{v}(T^ky)\|\\
& \geq C^{-1}\|\mathfrak{v}(T^kx)\| \geq C^{-2}\|\mathfrak{v}(T^kx)\|_\alpha = C^{-2},\end{align*}
which contradicts our assumption that $3C^6\theta^k < 1$. We conclude that $\lambda$ must be positive, and therefore
\[|1-\lambda| = \left|\left\|\mathfrak{v}(T^kx)\right\|_\alpha-\left\|\lambda \mathfrak{v}(T^ky)\right\|_\alpha\right|\leq \left\|\mathfrak{v}(T^kx)-\lambda \mathfrak{v}(T^ky)\right\|_\alpha < 3C^4\theta^k.\]
It follows easily that $\|\mathfrak{v}(T^kx)-\mathfrak{v}(T^ky)\|_\alpha <6C^4\theta^k<6C^6\theta^k$. The proof of the proposition is complete.
\end{proof}
Finally, we make use of the following simple result from matrix analysis, which we adapt from \cite[Lemma~2]{E}.
\begin{lemma}\label{elenmmna}
Let $\vvv \cdot \vvv$ be a norm on $\mathbb{R}^2$, let $B$ be a $2 \times 2$ matrix with $\vvv B\vvv \leq 1$, and suppose that $v \in \mathbb{R}^2$ with $\vvv v\vvv=1$. Let $C>1$ be any constant such that $C^{-1}\|u\| \leq \vvv u \vvv \leq C\|u\|$ for all $u \in \mathbb{R}^2$. Then $1-2C^2\sqrt{\vvv Bv -v \vvv} \leq \rho(B) \leq 1$.
\end{lemma}
\begin{proof}
If $M_1$ and $M_2$ are a pair of $2 \times 2$ real matrices, $\mu$ is an eigenvalue of $M_2$, and $\lambda_1,\lambda_2$ are the eigenvalues of $M_1$, then the bound
\[\min\{|\lambda_1-\mu|,|\lambda_2-\mu|\} \leq \sqrt{(\|M_1\|+\|M_2\|)(\|M_1-M_2\|)}\]
is well-known, see for example \cite[\S VIII]{Bhat}. Define $M_1:=B$ and $M_2:=B+\|v\|^{-2}(v-Bv)v^T$; we may then estimate
\[\|M_1-M_2\| =\|v\|^{-2}.\|(v-Bv)v^T\| \leq \|v\|^{-1}.\|Bv-v\| \leq C^2 \vvv Bv-v\vvv\]
and
 \[\|M_1\|+\|M_2\| \leq 2\|M_1\| + \|M_1-M_2\| \leq 2C^2 \vvv B\vvv + C^2\vvv Bv-v\vvv \leq 4C^2.\]
Since $M_2v=v$, $1$ is an eigenvalue of $M_2$, and it follows that $B$ has an eigenvalue $\lambda$ such that $|\lambda-1| \leq 2C^2\sqrt{\vvv Bv -v\vvv}$. The result follows in view of the elementary inequality $|\lambda| \leq \rho(B) \leq \vvv B \vvv =1$.
\end{proof}
The following key estimate forms the core of the proof of Theorem~\ref{main}(iii):

\begin{lemma}\label{fourpointsix}
Let $L \subset (0,1)$ be a compact interval, let $\gamma=p_n/q_n =[a_1,\ldots,a_n] \in L$ where $n,a_n>1$, and choose any $\alpha \in (0,\infty)$ such that $\mathfrak{r}(\alpha)=\gamma$. Let $p_{n-1}/q_{n-1}=[a_1,\ldots,a_{n-1}]$.  Then
\[
0 \leq (-1)^{n+1}\left(\frac{S(\gamma)- S(p_{n-1}/q_{n-1})}{\gamma-p_{n-1}/q_{n-1}}+\log \alpha\right) \leq K q_n\theta^{q_n},
\]
where $K>1$ and $\theta \in (0,1)$ are constants depending only on $L$.
\end{lemma}

\begin{proof}
From the classical theory of continued fractions we have $\gamma < p_{n-1}/q_{n-1}$  if $n$ is even, and the reverse inequality holds if $n$ is odd. By Proposition~\ref{sfunction} we know that $S$ is concave and $(-\log\alpha)$ is a subgradient of  $S$ at $\gamma$. Since the average gradient of a concave function on a closed interval is bounded below by every subgradient at the right endpoint, and bounded above by any subgradient at the left endpoint, we immediately deduce the inequality
\[
0 \leq (-1)^{n+1}\left(\frac{S(\gamma)- S(p_{n-1}/q_{n-1})}{\gamma-p_{n-1}/q_{n-1}}+\log \alpha\right).
\]
In proving the remainder of the lemma we will assume that $q_n \geq m$ for some constant $m \geq 1$ to be determined below. Indeed, given any such $m$ it is clear that $L$ contains only finitely many rational numbers $p_n/q_n$ with denominator less than $m$, and so by adjusting the constant $K>1$ if necessary, the full strength of the lemma follows from this special case.
Since the hypotheses of Proposition~\ref{cworp} are satisfied, we may fix an integer $k>\frac{1}{3}q_n$ and a point $x \in X_\gamma$ such that $d(x,T^{q_{n-1}}x) \leq 2^{-k}$ and $q_{n-1}^{-1}\log\rho(\mathcal{A}(T^kx,q_{n-1}))=S(p_{n-1}/q_{n-1})$. Let $y:=T^kx$. By Proposition~\ref{invtgraph} we have
\[\|\varrho(\alpha)^{-q_{n-1}}\mathcal{A}_\alpha(y,q_{n-1})\mathfrak{v}(y) - \mathfrak{v}(y)\|_\alpha = \|\mathfrak{v}(T^{q_{n-1}}y)-\mathfrak{v}(y)\|_\alpha \leq K \theta^k\]
for some constants $K>1$ and $\theta \in (0,1)$ depending on $L$. Combining this with Lemma~\ref{elenmmna} we obtain
\[1- 2C^2\sqrt{K}\theta^{\frac{1}{6}q_n} \leq \varrho(\alpha)^{-q_{n-1}}\rho(\mathcal{A}_\alpha(y,q_{n-1})) \leq 1.\]
where $C>1$ is the constant assigned by Lemma \ref{nf} to the compact set $\mathfrak{r}^{-1}(L) \subset (0,\infty)$. Since $\mathfrak{r}(\alpha)=\gamma$ we have $\varrho(\alpha)=e^{S(\gamma)}\alpha^\gamma$, and therefore
\[
1- 2C^2\sqrt{K}\theta^{\frac{1}{6}q_n} \leq\left(e^{-q_{n-1}S(\gamma)}\alpha^{-q_{n-1}\gamma}\right) \left(e^{S(p_{n-1}/q_{n-1})}\alpha^{p_{n-1}}\right) \leq 1.
\]
Let $m$ be an integer which is large enough that $1-2C^2\sqrt{K}\theta^{m/6}>e^{-1}$. By considering only those cases where $q_n \geq m$, we may by taking logarithms obtain
\[
-2C^2\sqrt{K}\theta^{\frac{1}{6}q_{n}} \leq q_{n-1}S(p_{n-1}/q_{n-1}) - q_{n-1}S(\gamma)  + (p_{n-1}-q_{n-1}\gamma)\log\alpha \leq 0,
\]
and by a slight rearrangement,
\[
0 \leq S(\gamma)-S(p_{n-1}/q_{n-1}) + \left(\gamma-\frac{p_{n-1}}{q_{n-1}}\right)\log\alpha \leq\frac{2C^2\sqrt{K}}{q_{n-1}}\theta^{\frac{1}{6}q_{n}}.
\]
Now, since
\[
(-1)^{n+1}\left(\gamma-\frac{p_{n-1}}{q_{n-1}}\right)= \frac{(-1)^{n+1}}{q_{n-1}q_n}\left(p_nq_{n-1}-q_np_{n-1}\right)= \frac{1}{q_{n-1}q_n}>0
\]
we may derive the inequality
\[
0 \leq (-1)^{n+1}\left(\frac{S(\gamma)- S(p_{n-1}/q_{n-1})}{\gamma-p_{n-1}/q_{n-1}}+\log \alpha\right) \leq 2q_nC^2\sqrt{K}\theta^{\frac{1}{6}q_n},
\]
which completes the proof.
\end{proof}

The following interesting result may be derived from Lemma~\ref{fourpointsix}:

\begin{corollary}\label{rationalbound}
Let $L \subset (0,1)$ be compact. Then there exist constants $C>1$, $\theta \in (0,1)$ depending on $L$ such that for all $p/q \in L$ with $p$ and $q$ coprime, the length of the interval $\mathfrak{r}^{-1}(p/q)$ is bounded by $Cq\theta^q$.
\end{corollary}

\begin{proof}
By enlarging the constant $C$ if necessary it is sufficient to consider rationals $p/q$ which are not of the form $1/k$ for an integer $k \geq 1$, since $L$ can contain only finitely many rationals of this form. Let $p/q \in L$ with $p$ and $q$ coprime and $p>1$. These assumptions allow us to find an integer $n \geq 2$ and integers $a_1,\ldots,a_n$ with $a_n>1$ such that $p/q=[a_1,\ldots,a_n]$. If $\alpha_1$ and $\alpha_2$ are the endpoints of the closed interval $\mathfrak{r}^{-1}(p/q)$, then
we may apply Lemma~\ref{fourpointsix} twice with $\alpha=\alpha_1, \alpha_2$ to see that $|\log \alpha_2 -\log\alpha_1| \leq 2Kq\theta^q$. The result follows.\end{proof}

Recall that if $\gamma=[a_1,a_2,\ldots] \in (0,1) \setminus \mathbb{Q}$ and the sequence $(a_n)$ is bounded, then there exists a constant $\delta>0$ such that $|\gamma-p/q|>\delta q^{-2}$ for all $q \in \mathbb{N}$ and $p \in \mathbb{Z}$ (see for example \cite{Kh}). In particular, for all sufficiently large $k$ the relation $|\gamma - p/q| \leq 1/q^k$ is impossible for integers $p \in \mathbb{Z}$ and $q \geq 2$. The proof of the following lemma is thus identical to the proof of \cite[Lemma~8.3]{HMST}:
\begin{lemma}\label{irrat}
Let $\gamma =[a_1,a_2,\ldots] \in (0,1) \setminus \mathbb{Q}$, and suppose that $a_n=1$ for all sufficiently large $n$. Then $\mathfrak{r}^{-1}(\gamma)$ is a singleton set.
\end{lemma}

\section{Preimages of irrational points: proof of Theorem.}
\label{sec:irrat}

Let us define $Z:=\mathfrak{r}^{-1}((0,1)\setminus \mathbb{Q})$, and partition $Z$ into two subsets as follows. We define $Z_1$ to be the set of all $\alpha \in Z$ such that the infinite continued fraction expansion $[a_1,a_2,\ldots]$ of the irrational number $\mathfrak{r}(\alpha) \in (0,1)$ satisfies $a_k=1$ for all but finitely many $k$. If $\gamma=[a_1,a_2,\ldots]$ is an irrational number of this type, then by Lemma~\ref{irrat} the set $\mathfrak{r}^{-1}(\gamma)$ is a singleton set. It follows that $Z_1$ is countable, and hence has zero Hausdorff dimension. Let us now define $Z_0:=Z \setminus Z_1$. Since the Hausdorff dimension of a countable union of sets is equal to the supremum of their individual Hausdorff dimensions, to prove that $\dim_H(Z)=0$ as claimed it is sufficient (and indeed necessary) to show that $\dim_H(Z_0)=0$. Moreover, it is sufficient to show that for some sequence of sets $L_k$ whose union covers $(0,1)$, each of the sets $Z_0 \cap \mathfrak{r}^{-1}(L_k)$ has Hausdorff dimension equal to zero. For the remainder of the proof, we fix a set $L$ of the form $[\frac{1}{k},1-\frac{1}{k}]$ with the aim of showing that the set $\mathfrak{r}^{-1}(L) \cap Z_0$ has Hausdorff dimension zero.

Given natural numbers $n, a_1,a_2,\ldots,a_n$, let $\Gamma_{(a_1,\ldots,a_n)}$ denote the half-open interval with endpoints $[a_1,\ldots,a_n]$ and $[a_1,\ldots,1+a_n]$ which excludes the former endpoint but includes the latter. This interval consists precisely of those elements of $(0,1)$ which admit a continued fraction expansion whose first $n$ entries are $a_1,\ldots,a_n$ respectively, and whose length is at least $n+1$. Given natural numbers $n \geq 1$ and $a_1,\ldots,a_n$, let us define
\[\mathcal{I}_{(a_1,\ldots,a_n)}:=\mathfrak{r}^{-1}\left(L \cap \Gamma_{(a_1,\ldots,a_n)}\right).\]
Note that by our choice of $L$, if $n \geq 2$ then the set $\mathcal{I}_{(a_1,\ldots,a_n)}$ is either empty or is equal to all of $\mathfrak{r}^{-1}(\Gamma_{(a_1,\ldots,a_n)})$. For each $N \geq 2$ let us define $\mathcal{U}_N$ to be the set of all $\mathcal{I}_{(a_1,\ldots,a_n)}$ such that $a_n>1$, $n \geq N$, and $a_k=1$ for all $k$ such that $N \leq k<n$. The reader may easily verify that $\Gamma_{(a_1,\ldots,a_n)} \cap \Gamma_{(b_1,\ldots,b_m)} = \emptyset$ when the vectors $(a_1,\ldots,a_n)$  and $(b_1,\ldots,b_m)$ are distinct, and furthermore,

\begin{equation}\label{boll}
\mathfrak{r}^{-1}(L) \cap Z_0 \subseteq \bigcup_{(a_1,\ldots,a_n) \in \mathcal{U}_N} \mathcal{I}_{(a_1,\ldots,a_n)}
\end{equation}
for every $N \geq 2$. We make the following key claim: if $n \geq 2$ is an integer, then for each $n$-tuple of natural numbers $(a_1,\ldots,a_n)\in\mathbb{N}^n$ such that $a_n>1$ there holds the inequality
\begin{equation}
\label{quay}
\diam \mathcal{I}_{(a_1,\ldots,a_n)} \leq Kq_n\theta^{q_n},
\end{equation}
where $p_n/q_n:=[a_1,\ldots,a_n]$ in least terms, and $K>1$ and $\theta \in (0,1)$ are constants depending only on $L$.

Let us prove this claim. Fix an integer $n \geq 2$ and suppose that $\mathcal{I}_{(a_1,\ldots,a_n)}$ is nonempty with $a_n>1$. Let $\alpha_1$ and $\alpha_2$ be respectively the infimum and the supremum of $\mathcal{I}_{(a_1,\ldots,a_n)}$, and let $\gamma_i=\mathfrak{r}(\alpha_i)$ for $i=1,2$. If $n$ is odd then we have
$\gamma_1=[a_1,\ldots,1+a_n]$ and $\gamma_2=[a_1,\ldots,a_n]$, and if $n$ is even then $\gamma_1=[a_1,\ldots,a_n]$ and $\gamma_2=[a_1,\ldots,1+a_n]$. Define also $p_{n-1}/q_{n-1}:=[a_1,\ldots,a_{n-1}] = \lim_{i \to \infty} [a_1,\ldots,a_{n-1},i]$. Our objective is to bound the difference $\alpha_2-\alpha_1$.

We consider first the case in which $n$ is odd, in which case $p_{n-1}/q_{n-1}<\gamma_1$. Recall that for a concave function defined on an interval $[a,b]$, the average gradient in the interval is greater than the value of any subgradient at $b$, and less than the value of any subgradient at $a$. Since $(-\log \alpha_2)$ is a subgradient of $S$ at $\gamma_2$, and $(-\log\alpha_1)$ is a subgradient of $S$ at $\gamma_1$, it follows that
\[
S(\gamma_2)-S(\gamma_1) \geq (\gamma_2-\gamma_1)(-\log\alpha_2)
\]
and
\[
S(\gamma_1)-S(p_{n-1}/q_{n-1}) \geq (\gamma_1-p_{n-1}/q_{n-1})(-\log \alpha_1).
\]
Adding these two inequalities together, we obtain
\[
(\gamma_2-\gamma_1)(-\log\alpha_2) +(\gamma_1-p_{n-1}/q_{n-1})(-\log\alpha_1)\leq S(\gamma_2)-S(p_{n-1}/q_{n-1})
\]
and therefore
\begin{align*}
(\gamma_1-p_{n-1}/q_{n-1})(\log \alpha_2 - \log \alpha_1)\leq &S(\gamma_2)-S(p_{n-1}/q_{n-1}) \\ &+ (\gamma_2-p_{n-1}/q_{n-1})\log\alpha_2.
\end{align*}
Since $L$ is a compact subinterval of $(0,1)$, $\mathfrak{r}^{-1}(L)$ is a compact subinterval of $(0,\infty)$, so there is a constant $C>0$ depending on $L$ such that $|\log x - \log y| \geq C^{-1}|x-y|$ for every $x,y \in L$. Hence
\begin{align*}
C^{-1}(\alpha_2-\alpha_1) &\leq \log \alpha_2 - \log \alpha_1\\ &\leq \left(\frac{\gamma_2-p_{n-1}/q_{n-1}}{\gamma_1-p_{n-1}/q_{n-1}}\right) \left(\frac{S(\gamma_2)-S(p_{n-1}/q_{n-1})} {\gamma_2-p_{n-1}/q_{n-1}}+\log\alpha_2\right).
\end{align*}
Let $p_{n-2}/q_{n-2}=[a_1,\ldots,a_{n-2}]$ in least terms. Since
\[
\gamma_1=\frac{(1+a_n)p_{n-1}+p_{n-2}}{(1+a_n)q_{n-1}+q_{n-2}}, \qquad\quad \gamma_2=\frac{a_n p_{n-1}+p_{n-2}}{a_n q_{n-1}+q_{n-2}},
\]
and $q_{n-1}p_{n-2}-q_{n-2}p_{n-1}=1$ it follows that
\begin{align*}
\frac{\gamma_2-p_{n-1}/q_{n-1}}{\gamma_1-p_{n-1}/q_{n-1}} &= \frac{1/(a_nq_{n-1}^2 + q_{n-1}q_{n-2})}{1/((1+a_n)q_{n-1}^2 + q_{n-1}q_{n-2})}\\
&=\frac{1+a_n + \frac{q_{n-2}}{q_{n-1}}}{a_n + \frac{q_{n-2}}{q_{n-1}}} \\
&\leq \frac{2+a_n}{a_n} \leq 3.
\end{align*}
Applying Lemma~\ref{fourpointsix}, we obtain
\begin{align*}
\alpha_2-\alpha_1  &\leq 3C(-1)^{n+1}\left(\frac{S(\gamma_2)-S(p_{n-1}/q_{n-1})} {\gamma_2-p_{n-1}/q_{n-1}}+\log\alpha_2\right)  \\
&\leq 3CKq_n\theta^{q_n}
\end{align*}
as required, which completes the proof of the claim in the case where $n$ is odd.

We now consider the case in which $n$ is even. In this case we have $p_{n-1}/q_{n-1}>\gamma_2$. By comparing subgradients in a similar manner to the odd case we arrive at the inequalities
\[
S(\gamma_2) -S(\gamma_1) \leq (\gamma_2-\gamma_1)(-\log \alpha_1),\]
\[S(p_{n-1}/q_{n-1})-S(\gamma_2) \leq (p_{n-1}/q_{n-1}-\gamma_2)(-\log\alpha_2).
\]
Adding these two inequalities yields
\[
S(p_{n-1}/q_{n-1})-S(\gamma_1) \leq (p_{n-1}/q_{n-1}-\gamma_2)(-\log\alpha_2) + (\gamma_2-\gamma_1)(-\log\alpha_1)
\]
and therefore
\[
S(p_{n-1}/q_{n-1})-S(\gamma_1) + (p_{n-1}/q_{n-1}-\gamma_1)\log \alpha_1 \leq (\gamma_2-p_{n-1}/q_{n-1})(\log \alpha_2-\log \alpha_1).
\]
Dividing by the negative real number $\gamma_2-p_{n-1}/q_{n-1}$ we obtain

\begin{align*}
\log \alpha_2-\log\alpha_1 &\leq -\left(\frac{\gamma_1-p_{n-1}/q_{n-1}}{\gamma_2-p_{n-1}/q_{n-1}}\right) \left(\frac{S(\gamma_1)-S(p_{n-1}/q_{n-1})}{\gamma_1-p_{n-1}/q_{n-1}}+\log \alpha_1\right)\\
&\leq 3(-1)^{n+1}\left(\frac{S(\gamma_1)-S(p_{n-1}/q_{n-1})} {\gamma_1-p_{n-1}/q_{n-1}}+\log \alpha_1\right),
\end{align*}
and it follows using Lemma~\ref{fourpointsix} that $\alpha_2-\alpha_1 \leq 3CKq_n\theta^{q_n}$ as before. This completes the proof of the claim.

We may now show directly that $\mathfrak{r}^{-1}(L) \cap Z_0$ has Hausdorff dimension zero. We recall the definition of the Hausdorff dimension of a set $Y \subseteq \mathbb{R}$. For each $\lambda \geq 0$, the $\lambda$-dimensional Hausdorff outer measure of the set $Y$ is defined to be the quantity
\[
\overline{\lim_{\delta \to 0}} \inf\left\{ \sum_{U \in \mathcal{U}} (\diam U)^\lambda \colon Y \subseteq \bigcup_{U \in \mathcal{U}} U \text{ and }\sup_{U \in \mathcal{U}} \diam U \leq \delta\right\},
\]
where each $\mathcal{U}$ is a collection of subsets of $\mathbb{R}$. The Hausdorff dimension of the set $Y$ is then defined to be the infimum of the set of all $\lambda \geq 0$ such that the $\lambda$-dimensional Hausdorff outer measure of $Y$ is zero, or equivalently the infimum of the set of all $\lambda \geq 0$ for which this value is finite.

Let $\lambda \in (0,1]$, and choose any $\delta>0$. We saw in \eqref{boll} that the union of the elements of $\mathcal{U}_N$ contains $\mathfrak{r}^{-1}(L) \cap Z_0$ for every $N \geq 2$. It follows from \eqref{quay} that if $N$ is large enough then every element of $\mathcal{U}_N$ has diameter less than $\delta$. For any such $N$ we have
\begin{align*}
\sum_{(a_1,\ldots,a_n) \in\mathcal{U}_N} \left(\diam \mathcal{I}_{(a_1,\ldots,a_n)}\right)^\lambda &\leq \sum_{(a_1,\ldots,a_n) \in\mathcal{U}_N} \left(Kq_n\theta^{q_n}\right)^{\lambda}\\
&< \sum_{p/q \in \mathbb{Q} \cap (0,1)} \left(Kq \theta^q\right)^\lambda\\
&= \sum_{q=2}^\infty \sum_{p=1}^{q-1} K^\lambda q^{\lambda} \theta^{\lambda q}\\
&< K\sum_{q=1}^\infty q^2 \theta^{\lambda q} = \frac{K\theta^\lambda(1+\theta^\lambda)}{(1-\theta^\lambda)^3},
\end{align*}
and since this bound is independent of $\delta$, we conclude that the $\lambda$-dimensional Hausdorff measure of $\mathfrak{r}^{-1}(L) \cap Z_0$ is finite. Since $\lambda$ may be chosen arbitrarily close to $0$, we conclude that $\dim_H(\mathfrak{r}^{-1}(L) \cap Z_0)=0$ as required. The proof of Theorem~\ref{main}(iii) is complete.

\section{Explicit formulae}
\label{sec:expl}

In this section we prove Theorem~\ref{exact-irrational} and present some bounds which can be used for practical computation of $\mathfrak{r}^{-1}(\gamma)$ in the special case of the matrices defined by \eqref{hmstmat} when $\ga$ is not too well approximated by rationals. In \cite{HMST} we proved the following result (the indexing of the sequences in the statement of Theorem~\ref{counter} has been adjusted so as to agree with the conventions used elsewhere in this section):
\begin{theorem}\label{counter}
Let $(\tau_n)_{n=0}^\infty$ denote the sequence of integers defined by $\tau_{-2}:=1$, $\tau_{-1},\tau_0:=2$, and \begin{equation}\label{eq:taun}
\tau_{n+1}:=\tau_n\tau_{n-1}-\tau_{n-2}\ \text{for all}\ n \geq 0,
\end{equation}
and let $(F_n)_{n=0}^\infty$ denote the sequence of Fibonacci numbers, defined by $F_0:=1$, $F_1:=1$ and $F_{n+1}:=F_n+F_{n-1}$ for all $n \geq 1$. For each $\alpha \geq 0$ let $\mathsf{A}_\alpha$ be the pair of matrices defined by \eqref{hmstmat}, and define a real number $\alpha_* \in (0,1]$ by
\begin{equation}\label{eq:alphastar}
\alpha_*:=\lim_{n \to \infty}
\left(\frac{\tau_n^{F_{n+1}}}{\tau_{n+1}^{F_n}}\right)^{(-1)^n}= \prod_{n=0}^\infty \left(1-\frac{\tau_{n-2}}{\tau_{n-1} \tau_n}\right)^{(-1)^{n+1} F_n}.
\end{equation}
Then this infinite product converges unconditionally, and $\mathsf{A}_{\alpha_*}$ does not have the finiteness property. The numerical value of the constant $\alpha_*$ is
\[\alpha_* \simeq 0.74932654633036755794396194809\ldots\]
\end{theorem}
Here $\alpha_*$ is in fact the unique positive real number such that $\mathfrak r(\alpha_*)=\left(3-\sqrt5\right)/2$. This particular constant was studied because $\gamma_*:=(3-\sqrt5)/2$ has a particularly simple continued fraction expansion: we have $\gamma_*=[2,1,1,1,1,\dots]$, which is the simplest possible expansion of an element of $(0,1/2)\setminus\mathbb{Q}$.

Now that Theorem~\ref{main} has been proved, the proof of Theorem~\ref{exact-irrational} may be obtained in a manner essentially similar to the proof of Theorem~\ref{counter}:
\begin{proof}[Proof of Theorem~\ref{exact-irrational}] Let $\gamma$, $(q_n)$, $(s_n)$ and $(\rho_n)$ be as in the statement of the theorem, and let $\mathfrak{r}$ be as in Definition~\ref{leppard}. An inductive argument as used in \S4 shows that $|s_n|_1=p_n$ and $|s_n|=q_n$ for every $n \geq 1$. In particular we have $\varsigma(s_n)=p_n/q_n$ for all $n \geq 1$ and therefore $q_n^{-1}\log \rho_n = S(p_n/q_n)$ for every positive integer $n$ by Proposition~\ref{sfunction}. By Theorem~\ref{main}(i) the function $\mathfrak{r}$ is continuous, monotone non-decreasing, and satisfies $\mathfrak{r}((0,+\infty)) \supseteq (0,1)$. In particular, $\mathfrak{r}^{-1}(\gamma)$ is nonempty, and is either a point or a closed interval. A consequence of Theorem~\ref{main}(iii) is that $\mathfrak{r}^{-1}(\gamma)$ has empty interior, and we conclude that there is a unique point $\alpha_\gamma \in (0,\infty)$ which satisfies $\mathfrak{r}(\alpha_\gamma)=\gamma$. It follows via Proposition~\ref{sfunction} that $-\log\alpha_\gamma \in \mathbb{R}$ is the unique subderivative of $S$ at $\gamma$, and hence $S$ is differentiable at $\gamma$ with $S'(\gamma)=-\log\alpha_\gamma$. We may therefore calculate
\begin{align*}
S'(\gamma)&=\lim_{n \to \infty}\frac{S\left(\frac{p_{n+1}}
{q_{n+1}}\right)-S\left(\frac{p_n}
{q_n}\right)}{\frac{p_{n+1}}
{q_{n+1}}-\frac{p_n}{q_n}}\\
&=\lim_{n \to \infty}\frac{\frac{1}{q_{n+1}}
\log\rho_{n+1}-\frac{1}{q_n}\log
\rho_n}{\frac{p_{n+1}}
{q_{n+1}}-\frac{p_n}{q_n}}\\
&=\lim_{n \to \infty}\frac{q_n \log\rho_{n+1}-q_{n+1}\log
\rho_n}{q_np_{n+1}-
q_{n+1}p_n}\\
&=\lim_{n \to
\infty}(-1)^n(q_n \log\rho_{n+1}-q_{n+1}\log
\rho_n),
\end{align*}
the existence of all of these limits being guaranteed by the differentiability of $S$ at $\gamma$. By rearranging we obtain
\[
\alpha_\gamma=\mathfrak r^{-1}(\ga)= e^{-S'(\gamma)} =  \lim_{n\to\infty}\left(\frac{\rho_n^{q_{n+1}}}
{\rho_{n+1}^{q_n}}\right)^{(-1)^n}
\]
as claimed. To derive the product expression for $\alpha_\gamma$ let us define
\[
\alpha_n :=\left(\frac{\rho_n^{q_{n+1}}}
{\rho_{n+1}^{q_n}}\right)^{(-1)^n}
\]
for each $n \geq -1$, and observe that
\[\frac{\alpha_n}{\alpha_{n-1}} = \left(\frac{\rho_n^{q_{n+1}}\rho_{n-1}^{q_n}}{\rho_{n+1}^{q_n} \rho_n^{q_{n-1}}}\right)^{(-1)^n} =\left(\frac{\rho_n^{a_{n+1}q_n}\rho_{n-1}^{q_n}}{\rho_{n+1}^{q_n}}
\right)^{(-1)^n} = \left(\frac{\rho_n^{a_{n+1}}\rho_{n-1}}{\rho_{n+1}}
\right)^{(-1)^nq_n}
\]
for each $n \geq 0$, where we have used the relation $q_{n+1}=a_{n+1}q_n+q_{n-1}$. We also have \[\alpha_{-1}=\frac{\rho_0^{q_{-1}}}{\rho_{-1}^{q_0}} = \frac{\rho(A_0)^0}{\rho(A_1)^1}=\frac{1}{\rho(A_1)}.\]
Hence
\begin{align*}\alpha_\gamma = \lim_{N \to \infty} \alpha_N &= \lim_{N \to \infty} \alpha_{-1} \prod_{n=0}^N \frac{\alpha_n}{\alpha_{n-1}}\\
& = \lim_{N \to \infty} \frac{1}{\rho(A_1)} \prod_{n=0}^N  \left(\frac{\rho_n^{a_{n+1}}\rho_{n-1}}{\rho_{n+1}}
\right)^{(-1)^nq_n}\\
&=  \frac{1}{\rho(A_1)} \prod_{n=0}^\infty \left(\frac{\rho_n^{a_{n+1}}\rho_{n-1}}{\rho_{n+1}}
\right)^{(-1)^nq_n}\end{align*}
as claimed. The proof is complete.\end{proof}

For the remainder of the section we let $\{\mathsf{A}_\alpha \colon \alpha \geq 0\}$ be the specific family of matrices defined by \eqref{hmstmat}. In this case we have $\rho(A_1)=1$, which means that the term $1/\rho(A_1)$ may be removed from the infinite product formula in Theorem~\ref{exact-irrational}. Let $\gamma \in (0,1)$ with infinite continued fraction expansion given by $\gamma=[a_1,a_2,a_3,\dots]$, and let $p_n/q_n$ be the $n$th convergent of $\gamma$. In view of the identity $A_0^T=A_1$, by replacing $\gamma$ with $1-\gamma$ and $\alpha$ with $1/\alpha$ if necessary, we will assume without loss of generality that $\alpha\in(0,1)$ and $\gamma\in(0,1/2)$, which is equivalent to $a_1 \geq 2$.

Let us consider the sequence of words specified by $\gamma$ given by $s_{-1}=1, s_0=0$, $s_1:=s_0^{a_1-1}s_{-1}$ and $s_{n+1}=s_n^{a_{n+1}}s_{n-1}$ for all $n\ge1$. Since $s_n$ prefixes $s_{n+1}$ for every $n$, it follows that $s_n$ prefixes $s_k$ for every $k \geq n$. Since furthermore the lengths $|s_n|=q_n$ tend to infinity, it follows that there is a unique infinite word $s_\infty \in \Sigma_2$  which is prefixed by every $s_n$. In particular this word is balanced, and it is recurrent: for each $n \geq 0$ the prefix $s_{n-1}$ occurs in at least two distinct locations in the prefix $s_{n+1}$, hence at least four distinct locations in the prefix $s_{n+3}$, and so forth, so that every subword of $s_\infty$ recurs in infinitely many positions.

Since $\varsigma(s_n)=p_n/q_n$, we have $\varsigma(s_n) \to \gamma$ as $n \to \infty$ and using Theorem~\ref{Sturm} it follows that $s_\infty \in X_\gamma$.  For $\gamma=\gamma_*$ the word $s_\infty$ is none other than the Fibonacci word $010010101001\dots$, which is the fixed point of the substitution $0\to01, 1\to0$ (see \cite{PF,Lot}).

Define $B_n=\mathcal A(s_n)$ for each $n \geq -1$. We have $B_{-1}=A_1, B_0=A_0, B_1=A_0^{a_1-1}A_1$, and
\begin{equation}\label{eq:rec}
B_{n+1}=B_n^{a_{n+1}}B_{n-1},\quad n\ge1.
\end{equation}
Put $\tau_n=\tr B_n$ and $\rho_n=\rho(B_n)$ as before. Note that since $\det B_n \equiv 1$ we have $\tau_n =\rho_n+\rho_n^{-1}$ and conversely $\rho_n = \frac{1}{2}(\tau_n + \sqrt{\tau_n^2-4})$. In particular $\tau_n\sim\rho_n$ as $n\to\infty$.
Subject to the above hypotheses we will prove the following rigorous estimate for the error in approximating $\alpha_\gamma$ by a partial product:

\begin{proposition}\label{rig}
Suppose there exists a constant $L>0$ and an integer $n_0 \geq 3$ such that
\begin{equation}\label{ineq:qn}
q_n \leq  L \rho_{n-1}\ \text{for all}\ n > n_0.
\end{equation}
Then for every $N \geq n_0$ there holds the inequality
\begin{equation}\label{bnd}\left|\log \alpha_\gamma - \log \left(\prod_{n=0}^N \left(\frac{\rho_n^{a_{n+1}}\rho_{n-1}} {\rho_{n+1}}\right)^{(-1)^nq_n}\right)\right| \leq \frac{2LC_0}{\rho_{N}},\end{equation}
where $C_0:=16(a_1+1)(a_1+2)+1$.
\end{proposition}

\begin{remark}
The assumption (\ref{ineq:qn}) is very weak. In particular, it holds for any non-Liouville $\gamma$ -- see Lemma~\ref{lem:nonliouville} below.
\end{remark}

In the special case where the continued fraction coefficients of $\gamma$ are bounded, Proposition~\ref{rig} lends itself to particularly easy verification. We have:

\begin{corollary}\label{rig2}
Suppose there exist integers $K \geq 2$ and $n_0 \geq 3$ and a constant $L>0$ such that the inequalities $q_{n_0+1} \leq L\rho_{n_0}$ and $\sup \{a_k \colon k \geq 2+n_0\} \leq K-1$ are satisfied, and such that the matrix $B_{n_0-1}-K\cdot I$ is non-negative, where $I$ denotes the identity. Then \eqref{bnd} holds for every $N \geq n_0$.
\end{corollary}

Since the spectral radii $\rho_n$ grow super-exponentially as a function of $n$ (see Lemma~\ref{lem:sec8} below), this allows very exact estimates to be made using relatively few terms. In order to prove the proposition and its corollary we require two lemmas. The following result is the technical core of the proof:

\begin{lemma}\label{lem:sec8}
The inequality
\begin{equation}\label{eq:ineqmain}
\left|1-\frac{\rho_{n+1}}{\rho_n^{a_{n+1}}\rho_{n-1}}\right| \le
\frac{C_0}{\rho_{n-1}^2}
\end{equation}
holds for all $n \geq 1$, where $C_0$ is as in Proposition~\ref{rig}. In particular, $\rho_{n+1}\sim \rho_n^{a_{n+1}}\rho_{n-1}$ as $n\to\infty$.
\end{lemma}

\begin{proof}
We first construct an auxiliary continued fraction as follows:
\[
\beta=[d_1,1,d_2,1,d_3,1,\dots],
\]
where $d_1=a_1-1$ and $d_k$ is the number of zeros between the $k$th and $(k+1)$st unities in $s_\infty$ for $k\ge2$. For instance, for $a_1=2$ and $a_k\equiv1$ for $k\ge2$ (i.e., the Fibonacci word $s_\infty$) we have $\beta=[1,1,2,1,1,1,2,\dots]$. We denote
\[
\beta=[b_1,b_2,\dots].
\]
Note that since the number of consecutive zeroes in $s_\infty$ is bounded by $a_1$ (see, e.g., \cite{Lot}), we have $b_k\le a_1$  for all $k$.

Let $u_0=1$, and for each $n \geq 1$ let $u_n$ denote the length of the word constructed
from $s_n$ by replacing every string of consecutive zeros with a single zero. That is, $s_1=0^{a_1-1}1$, whence $u_1=2$;
$s_2=(0^{a_1-1}1)^{a_2}0$, whence $u_2=2a_2+1$, etc. Define also
\[
\frac{P_k}{Q_k}=[b_1,b_2,\dots, b_k].
\]
Recall the following well known relation between matrix products involving powers of $A_0, A_1$, and continued fractions:
\[
A_1^{a_m}A_0^{a_{m-1}}\cdots A_1^{a_1} =
\begin{pmatrix}1&0\\a_m&1 \end{pmatrix}
\begin{pmatrix}1&a_{m-1}\\0&1 \end{pmatrix}
\cdots
\begin{pmatrix}1&0\\a_1&1 \end{pmatrix}
=
\begin{pmatrix}p_m&p_{m-1}\\q_m&q_{m-1} \end{pmatrix}
\]
if $m$ is odd, and
\[
A_1^{a_m}A_0^{a_{m-1}}\cdots A_1^{a_1} =
\begin{pmatrix}1&a_m\\0&1 \end{pmatrix}
\begin{pmatrix}1&0\\a_{m-1}&1 \end{pmatrix}
\cdots
\begin{pmatrix}1&0\\a_1&1 \end{pmatrix}
=
\begin{pmatrix}p_{m-1}&p_m\\q_{m-1}&q_m
\end{pmatrix}
\]
if $m$ is even (see, e.g., \cite{Frame}). Hence we have
\begin{equation}\label{eq:Bn}
B_n = \begin{cases}
\begin{pmatrix} P_{u_n} & P_{u_n-1}\\  Q_{u_n} & Q_{u_n-1}
\end{pmatrix} & u_n\ \text{is odd} \\
\begin{pmatrix} P_{u_n-1} & P_{u_n}\\  Q_{u_n-1} & Q_{u_n}
\end{pmatrix} & u_n\ \text{is even},
\end{cases}
\end{equation}
Let us compute the eigenvectors for $B_n$ of the first type:
\begin{equation}\label{eq:syst}
\begin{pmatrix} P_{u_n} & P_{u_n-1}\\  Q_{u_n} & Q_{u_n-1}
\end{pmatrix}\begin{pmatrix} \xi_n \\ 1\end{pmatrix}=
\la_n \begin{pmatrix} \xi_n \\ 1\end{pmatrix},
\end{equation}
where $\la_n=\rho_n$ or $\rho_n^{-1}$. Solving this system, we get a quadratic equation:
\[
Q_{u_n}\xi_n^2+(Q_{u_n-1}-P_{u_n})\xi_n-P_{u_n-1}=0.
\]
Dividing it by $Q_{u_n}$, we obtain
\begin{equation}\label{eq1}
\xi_n^2+
\left(\frac{Q_{u_n-1}}{Q_{u_n}}-\frac{P_{u_n}}{Q_{u_n}}\right)
\xi_n -\frac{P_{u_n-1}}{Q_{u_n-1}}\cdot\frac{Q_{u_n-1}}{Q_{u_n}}
=0,
\end{equation}
whence
\begin{equation}\label{eq2}
\left(\xi_n+\frac{Q_{u_n-1}}{Q_{u_n}}\right)(\xi_n-\beta)=
\xi_n\left(\frac{P_{u_n}}{Q_{u_n}}-\beta\right) +
\frac{Q_{u_n-1}}{Q_{u_n}}
\left(\frac{P_{u_{n-1}}}{Q_{u_{n-1}}}-\beta\right).
\end{equation}
Let from here on $\xi_n$ stand for the positive root of (\ref{eq1}). From (\ref{eq:syst}) it follows that $Q_{u_n}\xi_n+Q_{u_n-1}=\rho_n<\tau_n=P_{u_n}+Q_{u_n-1}$, whence $\xi_n<\frac{P_{u_n}}{Q_{u_n}}<1$.

Since the $b_k$ are bounded, we have
\begin{equation}\label{ineq:Qk}
\frac{Q_k}{Q_{k-1}}=b_k+\frac1{b_{k+1}+\dots}\le b_k+1\le a_1+1.
\end{equation}
Hence from (\ref{eq2})
\[
|\xi_n-\beta|\le (a_1+1)\cdot\left(\left|\frac{P_{u_n}}{Q_{u_n}}-\beta\right|+ \left|\frac{P_{u_n-1}}{Q_{u_n-1}}-\beta\right|\right).
\]
By (\ref{ineq:Qk}), we have
\begin{equation}\label{eq:laP}
\left|\beta-\frac{P_{u_n}}{Q_{u_n}}\right|\le \frac{1}{Q_{u_n} {Q_{u_n+1}}}\le \frac{1}{Q_{u_n}^2},\quad \left|\beta-\frac{P_{u_n-1}}{Q_{u_n-1}}\right|\le \frac{1}{Q_{u_n} {Q_{u_n-1}}}\le \frac{a_1+1}{Q_{u_n}^2}.
\end{equation}
Hence
\[
|\beta-\xi_n|\le \frac{(a_1+1)(a_1+2)}{Q_{u_n}^2}.
\]
Since $\rho_n<\tau_n=P_{u_n}+Q_{u_n-1}$, we have $\rho_n<2Q_{u_n}$, whence
\begin{equation}\label{eq:est1}
|\beta-\xi_n|\le \frac{C_1}{\rho_n^2},
\end{equation}
where
\[
C_1=4(a_1+1)(a_1+2).
\]
(In the case of even $u_n$, we have $\tau_n = P_{u_n-1}+Q_{u_n}<2Q_{u_n}$, so (\ref{eq:est1}) holds as well.)

Let $\xi_n'<0$ denote the other solution of (\ref{eq1}). Put
\[
D_n = \begin{pmatrix} \xi_n & \xi_n' \\ 1 & 1 \end{pmatrix}.
\]
We have
\[
D_n^{-1}B_nD_n = \begin{pmatrix} \rho_n & 0 \\ 0 & \rho_n^{-1}
\end{pmatrix}.
\]
We want to apply
the change of coordinates given by $D_n$ to the equation
(\ref{eq:rec}) and then obtain a relation for the traces. Since
$\tr(D_n^{-1}B_{n+1}D_n)=\tr B_{n+1}=\tau_{n+1}=\rho_{n+1}+\rho_{n+1}^{-1}$, we will be only concerned with estimating $\tr(D_n^{-1}B_{n-1}D_n)$.

Assume that $u_{n-1}$ is even; then
\[
B_{n-1}=\begin{pmatrix} P_{u_{n-1}-1} & P_{u_{n-1}}\\
Q_{u_{n-1}-1} & Q_{u_{n-1}}
\end{pmatrix}.
\]
(The case of odd $u_{n-1}$ is completely analogous.) We have
\begin{align*}
D_n^{-1}B_{n-1}D_n &= \frac1{\xi_n-\xi_n'}\begin{pmatrix} 1 &
-\xi_n' \\ -1 & \xi_n \end{pmatrix}\begin{pmatrix} P_{u_{n-1}-1} &
P_{u_{n-1}}\\  Q_{u_{n-1}-1} & Q_{u_{n-1}}\end{pmatrix}
\begin{pmatrix} \xi_n & \xi_n' \\ 1 & 1 \end{pmatrix}
\\
&=\begin{pmatrix} \rho_{n-1}-r_{n-1} & \dots \\ \dots & r_{n-1} \end{pmatrix},
\end{align*}
where
\begin{equation}\label{eq:rn-1}
r_{n-1}=\xi_n'(\xi_nQ_{u_{n-1}-1}-P_{u_{n-1}-1})+\xi_n
Q_{u_{n-1}}-P_{u_{n-1}}.
\end{equation}
By (\ref{eq:laP}) and (\ref{eq:est1}),
\begin{align*}
|\xi_n Q_{u_{n-1}}-P_{u_{n-1}}| & \le
|\xi_n-\beta|Q_{u_{n-1}}+ |\beta Q_{u_{n-1}}-P_{u_{n-1}}| \\
&\le \frac{C_1\cdot Q_{u_{n-1}}}{\rho_n^2}+\frac1{Q_{u_{n-1}+1}} \\
&\le \frac{C_1\rho_{n-1}}{\rho_n^2}+\frac{2}{\rho_{n-1}} \\
&\le \frac{2C_1}{\rho_{n-1}},
\end{align*}
in view of $C_1>2, \rho_n>\rho_{n-1}$. Since $ |\beta Q_{u_{n-1}-1}-P_{u_{n-1}-1}|\le Q_{u_{n-1}}^{-1}$, we have the same bound for $|\xi_nQ_{u_{n-1}-1}-P_{u_{n-1}-1}|$, whence from (\ref{eq:rn-1}), in view of $|\xi_n'|<1$,
\begin{equation}\label{eq:rn-bound}
r_{n-1}\le \frac{4C_1}{\rho_{n-1}}.
\end{equation}
By our construction,
\[
D_n^{-1}B_n^{a_{n+1}}D_n=\begin{pmatrix} \rho_n^{a_{n+1}} & 0 \\ 0
& \rho_n^{-a_{n+1}} \end{pmatrix},
\]
whence
\[
D_n^{-1} B_{n+1}D_n = \begin{pmatrix} \rho_n^{a_{n+1}} & 0 \\ 0
& \rho_n^{-a_{n+1}} \end{pmatrix} \begin{pmatrix} \rho_{n-1}-r_{n-1} & \dots \\ \dots & r_{n-1} \end{pmatrix}.
\]
Taking the traces yields
\[
\tau_{n+1}=\rho_n^{a_{n+1}}(\tau_{n-1}-r_{n-1})+
\rho_n^{-a_{n+1}}r_{n-1}.
\]
Using $\tau_n=\rho_n+\rho_n^{-1}$, we obtain
\[
\rho_{n+1}+\rho_{n+1}^{-1}= \rho_n^{a_{n+1}}(\rho_{n-1}+\rho_{n-1}^{-1}-r_{n-1}) + \rho_n^{-a_{n+1}}r_{n-1}.
\]
Therefore,
\[
1-\frac{\rho_{n+1}}{\rho_n^{a_{n+1}}\rho_{n-1}} =\frac{r_{n-1}}{\rho_{n-1}}+ \frac1{\rho_{n+1}\rho_n^{a_{n+1}}\rho_{n-1}} - \frac1{\rho_{n-1}^2}-\frac{r_{n-1}}{\rho_n^{2a_{n+1}}\rho_{n-1}},
\]
whence
\[
1-\frac{\rho_{n+1}}{\rho_n^{a_{n+1}}\rho_{n-1}}\ge - \frac1{\rho_{n-1}^2}-\frac{r_{n-1}}{\rho_n^{2a_{n+1}}\rho_{n-1}} \ge -\frac{C_0}{\rho_{n-1}^2}
\]
(in view of $r_{n-1}/\rho_n^{2a_{n+1}}<1$ and $C_0>2$),
and
\begin{align*}
1-\frac{\rho_{n+1}}{\rho_n^{a_{n+1}}\rho_{n-1}} &\le\frac{r_{n-1}}{\rho_{n-1}}+ \frac1{\rho_{n+1}\rho_n^{a_{n+1}}\rho_{n-1}} \\
&< \frac{4C_1}{\rho_{n-1}^2}+\frac1{\rho_{n-1}^3}\\
&<\frac{4C_1+1}{\rho_{n-1}^2}=\frac{C_0}{\rho_{n-1}^2}.
\end{align*}
\end{proof}

We also require the following lower estimate on the growth of the sequence $(\rho_n)$.

\begin{lemma}\label{ghu}
If $n \geq 1$ and $K \geq 1$ are integers such that the matrix $B_{n-1}-K\cdot I$ is non-negative, then $\rho_{n+1} \geq K \rho_n$.
In particular we have $\rho_{n+1} \geq 2\rho_n$ for all $n \geq 3$.
\end{lemma}
\begin{proof}
Given a pair of matrices $A$ and $B$ we will use the notation $A \geq B$ to mean that the difference $A-B$ is a non-negative matrix. If $A \geq B$ then obviously also $\tr A \geq \tr B$, and $AC \geq BC$ and $CA \geq CB$ for any non-negative matrix $C$. Note in particular that $A_0,A_1 \geq I$, and hence if $C$ is any product of powers of $A_0$ and $A_1$ then $C \geq I$. It follows that $B_{n+1} \geq B_{n}$ for all $n \geq 0$.

If $n, K \geq 1$ and $B_{n-1} \geq K\cdot I$, then $B_{n+1} = B_n^{a_{n+1}}B_{n-1} \geq KB_n^{a_{n+1}} \geq KB_n$ and therefore $\tau_{n+1} \geq K \tau_n$. It follows that
\[\rho_{n+1} = \frac{1}{2} \left(\tau_{n+1} + \sqrt{(\tau_{n+1})^2 -4}\right) \geq \frac{K}{2}\left(\tau_n + \sqrt{(\tau_n)^2 -\frac{4}{K^2}}\right) \geq K\rho_n\]
as required.
Since $a_1 \geq 2$, we may estimate
\begin{align*}
B_2 &= B_1^{a_2}B_0 = (B_0^{a_1-1}B_{-1})^{a_2}B_0 = (A_0^{a_1-1}A_1)^{a_2} A_0 \\
&\geq A_0A_1A_0 = \begin{pmatrix}2 & 3 \\ 1 & 2\end{pmatrix} \geq 2I,
\end{align*}
and since $B_{n-1} \geq B_2$ for all $n \geq 3$ it follows that $\rho_{n+1} \geq 2\rho_n$ for all $n \geq 3$ as claimed.
\end{proof}
We may now give the proofs of Proposition~\ref{rig} and Corollary~\ref{rig2}.
\begin{proof}[Proof of Proposition~\ref{rig}.]
Let $N \geq n_0$ and define
\[
\alpha_N :=
\prod_{n=0}^N
\left(\frac{\rho_n^{a_{n+1}}\rho_{n-1}}{\rho_{n+1}}
\right)^{(-1)^nq_n}.
\]
Using (\ref{eq:ineqmain}) together with the second clause of Lemma~\ref{ghu},
\begin{align*}
|\log\alpha_\gamma-\log\alpha_N|&=\left|\sum_{n=N+1}^\infty
(-1)^nq_n \log\left(\frac{\rho_n^{a_{n+1}}\rho_{n-1}}{\rho_{n+1}}
\right)\right| \\
&\le \sum_{n=N+1}^\infty q_n \left|\log\left(\frac{\rho_{n+1}}{\rho_n^{a_{n+1}}\rho_{n-1}}
\right)\right| \\
&\le \sum_{n=N+1}^\infty q_n \left|1-\frac{\rho_{n+1}}{\rho_n^{a_{n+1}}\rho_{n-1}}
\right| \\
& \le C_0 \sum_{n=N+1}^\infty \frac{q_n}{\rho_{n-1}^2}\\
& \le LC_0 \sum_{n=N+1}^\infty \frac{1}{\rho_{n-1}}\\
& \le LC_0 \sum_{n=N+1}^\infty \frac{1}{2^{N+1-n}\rho_N}\\
&=\frac{2LC_0}{\rho_N}
\end{align*}
as required.
\end{proof}
\begin{proof}[Proof of Corollary~\ref{rig2}.]
For each $j \geq 1$ we have
\[q_{n_0+1+j} = a_{n_0+1+j} q_{n_0+j} + q_{n_0+j-1} \leq K q_{n_0+j},\]
and it follows that $q_{n_0+1+j} \leq K^j q_{n_0+1}$ for all $j\geq 0$. On the other hand since $B_{n_0-1}-KI$ is non-negative, $B_{n_0+j-2}$ is non-negative for every $j \geq 1$, and using Lemma \ref{ghu} we deduce that $\rho_{n_0+j} \geq K \rho_{n_0+j-1}$ for all such $j$. We therefore have $q_{n_0+1+j} \leq K^j q_{n_0+1} \leq K^jL\rho_{n_0} \leq L\rho_{n_0+j}$ for all $j \geq 0$ and we may apply Proposition~\ref{rig}.
\end{proof}

Let us show that the hypothesis $q_n=O(\rho_{n-1})$ is valid for ``typical'' $\gamma$ in a suitable sense:

\begin{lemma}\label{lem:nonliouville}
If $\ga$ is not Liouville, then $q_n \leq \rho_{n-1}$ for all sufficiently large $n$.
\end{lemma}
\begin{proof}
Since $\ga$ is not Liouville, there exists $\delta>0$
such that
\[
\left|\ga-\frac{p_n}{q_n}\right|\ge\frac1{q_n^{\delta+1}}.
\]
Since
\[
\left|\ga-\frac{p_n}{q_n}\right|\le\frac1{q_n q_{n+1}},
\]
we have $q_{n+1}\le q_n^{\delta}$. Thus, it suffices to show that $q_n^\delta\le \rho_n$ for $n$ large enough. By Lemma~\ref{lem:sec8}, $\rho_n\sim \rho_{n-1}^{a_n}\rho_{n-2}$, whence $\log\rho_n\sim a_n\log\rho_{n-1}+\log\rho_{n-2}$. Consequently, $\log\rho_n\ge \text{const}\cdot q_n$. (Since $q_n=a_nq_{n-1}+q_{n-2}$.) Now the claim follows from the fact that the $q_n$ grow at least exponentially fast, whence $\log q_n \ll q_n$.
\end{proof}

In fact the upper bound $q_n=O(\rho_{n-1})$ holds for ``most'' Liouville numbers as well. Effectively, if this inequality fails, this means that $a_n > A^{a_{n-1}}$ infinitely often for some constant $A>1$, which is an exceptionally strong condition.

\begin{remark}
It is natural to ask whether a formula like (\ref{eq:alphastar}) -- with traces instead of spectral radii -- holds in a more general case of irrational $\ga$ (instead of (\ref{eq:explic}), where the multipliers are irrational). The answer is yes -- provided, for example, the condition $q_n=O(\rho_{n-1})$ holds. Indeed, this condition implies
\[
\left(1+\frac1{\rho_n^2}\right)^{q_{n+1}}\to0\quad n\to\infty,
\]
whence we can replace the spectral radii with the corresponding traces so as to obtain
\begin{equation}\label{eq:traces}
\alpha_\gamma =
\lim_{n\to\infty}\left(\frac{\tau_n^{q_{n+1}}}
{\tau_{n+1}^{q_n}}\right)^{(-1)^n}  = \prod_{n=0}^\infty
\left(\frac{\tau_n^{a_{n+1}}\tau_{n-1}}{\tau_{n+1}}
\right)^{(-1)^nq_n}.
\end{equation}
Note that if the $a_n$ grow extremely fast (for instance, if $a_n=q_{n-1}$), then (\ref{eq:traces}) is false; one can show that if it were true, then $\mathfrak r^{-1}(\gamma)$ would be an interval, contradicting Theorem~\ref{main} (iii).

For $\ga=\frac{3-\sqrt5}2$ the formula (\ref{eq:traces}) is exactly  (\ref{eq:alphastar}), in view of the recurrence relation~(\ref{eq:taun}). Indeed, we have $q_n=F_n$ and
\[
\prod_{n=0}^\infty
\left(\frac{\tau_n\tau_{n-1}}{\tau_{n+1}}
\right)^{(-1)^nF_n}= \prod_{n=0}^\infty
\left(1-\frac{\tau_{n-2}}{\tau_{n-1}\tau_n}
\right)^{(-1)^{n+1}F_n}.
\]
\end{remark}

\begin{remark}
Despite having such a fast convergent infinite product for $\alpha_\gamma$, we still cannot use it to claim that $\alpha_\gamma$ is irrational if $\gamma$ is irrational. Such a result would show that the family~(\ref{hmstmat}) does not contain a counterexample to the rational finiteness conjecture (see \cite{BJ} for more detail).
\end{remark}

\begin{remark}Another natural question is whether there exists a recurrence relation -- or rather a sequence of such relations -- for the $\tau_n$ in the case of a general irrational $\ga$. It can be shown that if $a_n$ and $a_{n+1}$ are fixed, then there will be the same recurrence relation for $\tau_{n+1}$, irrespective of the rest of $a_k$. However, even in the simple case $a_n=a_{n+1}=2$, for instance, we have the relatively unstraightforward identity
\[
\tau_{n+1}=\tau_n^2\tau_{n-1}-\frac{\tau_n^2}{\tau_{n-1}}- \frac{\tau_n\tau_{n-2}}{\tau_{n-1}}-\tau_{n-1}.
\]
And for larger $a_n$ and $a_{n+1}$, it becomes messier, though the two most significant terms are always $\tau_n^{a_{n+1}}\tau_{n-1}-\tau_n^{a_{n+1}}/\tau_{n-1}$, provided the $a_n$ do not grow too fast. The authors are grateful to Kevin Hare for helping them with these computations.
\end{remark}

The following examples yield new explicit parameters $\alpha$ such that the system $\{A_0, \alpha A_1\}$ does not possess the finiteness property:

\begin{example}Put $\ga=\sqrt5-2$. It is algebraic and therefore, not Liouville. Here $\alpha=0.4596704785\dots$
\end{example}

\begin{example}Put $\ga=\sqrt[3]2-1=[3,1,5,1,1,4,1,1,8,1,\dots]$. Here $\alpha=0.5587336687\dots$
\end{example}

\begin{example}As is well known, $e-2=[1,2,1,1,4,1,1,6,1,1,8,\dots]$, which implies that $e$ is not Liouville. Put
\[
\ga=\frac{e-2}{e-1}=0.4180232931\ldots= [2,2,1,1,4,1,1,6,1,1,8,\dots].
 \]
Here $\alpha=0.7904851693\dots$
\end{example}

\bigskip
\footnotesize
\noindent\textit{Acknowledgments.}

The proof of the impossibility of the equation \eqref{contro} was facilitated by discussions which took place on the MathOverflow website. The authors would like to thank I.~Agol and Q.~Yuan for helpful conversations pertaining to this proof, and the administrators of the MathOverflow website for making these interactions possible.

The authors are indebted to Kevin Hare for many stimulating discussions and insights.

Ian Morris was supported as a Postdoctoral Research Fellow by the ERC grant MALADY (AdG 246953).

\bibliographystyle{siam}
\bibliography{DSJSR}

\end{document}